\documentclass[10pt, a4paper]{article}
\usepackage[T1]{fontenc}
\usepackage[utf8]{inputenc}
\usepackage[a4paper]{geometry}
\usepackage[french,english]{babel}
\usepackage{amsmath}
\usepackage{amsfonts}
\usepackage{amssymb}
\usepackage{amsthm}
\usepackage{mathtools}
\usepackage{esint}
\usepackage{hyperref}

\allowdisplaybreaks

\newtheorem{lemma}{Lemma}[section]
\newtheorem{proposition}{Proposition}[section]
\newtheorem{theorem}{Theorem}[section]
\newtheorem{corollary}{Corollary}[section]
\newtheorem{definition}{Definition}[section]
\newtheorem*{remark}{Remark}

\DeclareMathOperator{\card}{card}

\newcommand{\R}{\mathbf{R}}
\newcommand{\C}{\mathbf{C}}
\newcommand{\N}{\mathbf{N}}
\newcommand{\Z}{\mathbf{Z}}

\newcommand{\Continuous}{\mathcal{C}}

\newcommand{\Smooth}[1]{\mathcal{C}^\infty(#1)}
\newcommand{\SmoothComp}[1]{\mathcal{C}^\infty_0(#1)}
\newcommand{\charaset}{\chi}

\newcommand{\collection}[1]{\mathcal{#1}}
\newcommand{\diff}{\mathop{}\mathrm{d}}

\newcommand{\goesto}{\rightarrow}
\newcommand{\Vol}[2]{\mu\left(B\left(#1,#2\right)\right)}
\renewcommand{\div}{\mathrm{div}}
\newcommand{\llangle}{\left\langle}
\newcommand{\rrangle}{\right\rangle}
\newcommand{\doubling}[2][\eta]{(\mathbf{D})^{#1}_{#2}}
\DeclarePairedDelimiter{\set}{\{}{\}}
\newcommand{\parenfrac}[2]{\left(\frac{#1}{#2}\right)}
\newcommand{\dyadic}{\collection{D}}

\newcommand{\rd}[2][\nu]{(\mathbf{RD})^{#1}_{#2}}
\newcommand{\FK}[2][\eta]{(\mathbf{RFK})^{#1}_{#2}}

\newcommand{\Kcond}[1]{(\mathbf{K})_{#1}}

\numberwithin{equation}{section}

\title{Lower bound of Schrödinger operators on Riemannian manifolds}
\author{M. LANSADE}
\date{}

\begin{document}

\begin{abstract}
	We show that a complete weighted manifold which satisfies to a relative Faber Krahn inequality admits a trace inequality for the measure with density $V$, with the constant depending on a Morrey norm of $V$. From this we obtain estimates on the lower bound of the spectrum of the Schrödinger operators with potential $V$ and positivity conditions for such operators. It also yields a $L^2$ Hardy inequality.
\end{abstract}

\maketitle

\section{Introduction}

In \cite{FeffermanPhongConf82, Fefferman83} Fefferman and Phong established the inequality, for $p > 1$:
\begin{equation}\label{eq:feff_phong_ineq}
	\int_{\R^n} V(x) \psi(x)^2 \diff x \leq C_{n,p} N_p(V) \int_{\R^n} |\nabla \psi(x)|^2 \diff x,
\end{equation}

\noindent for any $\psi$ smooth with compact support, where $V$ is a non negative and locally integrable function, $C_{n,p}$ is a constant depending only on the dimension and $p$, and $N_p$ is the Morrey norm:
\begin{equation}
	N_p(V) =\sup_{\substack{x\in \R^n\\ r > 0}} \left( r^{2p-n} \int_{B(x,r)} |V(y)|^p \diff y\right)^{1/p}.
\end{equation} 

Such an inequality yields a positivity condition for the Schrödinger operator $H = \Delta - V$ (with $\Delta = -\sum_{i = 1}^n \partial_i^2$), namely that if $N_p(V) \leq 1/C_{n,p}$, then $H$ is a positive operator. In fact they also gave the following estimates on the lower bound of the spectrum of $H$, $\lambda_1(H)$:

\begin{multline}\label{eq:lower bound estimate R^n}
	\sup_{\substack{x \in \R^n\\ r > 0}} \left( C_1 r^{-n}\int_{B(x,r)} V \diff y - r^{-2} \right) \leq -\lambda_1(H) \\
	-\lambda_1(H) \leq \sup_{\substack{x \in \R^n\\ r > 0}} \left( C_p  \left(r^{-n}\int_{B(x,r)} V^p \diff y\right)^{1/p} - r^{-2} \right).
\end{multline}

The conditions for inequalities such as \eqref{eq:feff_phong_ineq} (though with a constant that doesn't necessarily depends on the Morrey norm) to hold in $\R^n$ has been studied extensively, see for example in \cite{ChangWilsonWolff85, KermanSawyer86,Maz'yaVerbitsky95}. And in \cite{Maz'yaVerbitsky02}, Maz'ya and Verbitsky establish necessary and sufficient conditions for an analoguous inequality to \eqref{eq:feff_phong_ineq} to hold with complex valued $V$. That being the case, it seems interesting to study to what extent, and under which geometrical hypotheses those results extend on other spaces, such as Riemannian manifolds.

The first aim of this article is to generalize the results of Fefferman and Phong to a weighted Riemannian manifold $M$. A natural way to do that would be to use the Poincaré inequality: for any $\kappa > 1$, there is a constant $C >0$, such that for all $x\in M$, $r > 0$, and for any $f \in \Smooth{B(x,\kappa r)}$: 
\begin{equation*}
	\int_{B(x,r)} | f - f_{B(x,r)}|\diff\mu \leq C r \int_{B(x,\kappa r)} |\nabla f|\diff\mu, 
\end{equation*}

\noindent where $f_B = \frac{1}{\mu(B)}\int_B f \diff\mu$. It turns out that the result still holds under some weaker hypothesis. Our proof will follow the general idea used by Schechter in \cite{Schechter89}, that \eqref{eq:feff_phong_ineq} follows from the inequality (which holds in $\R^n$ following a result of Muckenhoupt and Wheeden \cite{MuckenhouptWheeden74}):

\begin{equation}\label{eq:Riesz_Max_L2}
	\left\|I_1 f\right\|_{L^2} \leq C \|M_1 f\|_{L^2},
\end{equation}

\noindent with 
\begin{equation}
	I_1 f(x) = c_n \int_{\R^n} \frac{f(y)}{|x-y|^{n-1}} \diff\mu(y),\quad M_1 f(x) = \sup_{r> 0} r^{1-n} \int_{B(x,r)}|f(y)|\diff y,
\end{equation} 

\noindent and that \eqref{eq:lower bound estimate R^n} is proved using similar estimates, with $(\Delta + \lambda^2)^{-1/2}$ replacing $I_1$.

The proof of the generalisation of \eqref{eq:lower bound estimate R^n} will naturally yields weak versions of \eqref{eq:feff_phong_ineq}, which holds under weaker hypothesis.

\subsection{Definitions and Notations}

A weighted Riemannian manifold $(M,g,\mu)$, or simply a weighted manifold, is the data of a smooth manifold $M$, $g$ a smooth Riemannian metric on $M$, and a Borel measure $\diff\mu = \sigma^2 \diff v_g$ on $M$, with $\sigma$ a smooth positive function on $M$ and $v_g$ is the Riemannian volume measure associated with the metric $g$. We define the (weighted) Dirichlet Laplace operator as the Friedrichs extension of the operator on $\SmoothComp{M}$ defined by:

\begin{equation*}
	\Delta_\mu f = -\sigma^{-2} \div(\sigma^2 \nabla f),
\end{equation*}
\noindent with associated quadratic form $Q(\psi) = \int_M |\nabla\psi|^2 \diff\mu$. We will usually write the Dirichlet Laplace operator as simply $\Delta$. 

On a metric space $(X,d)$, for $x\in X$, $r > 0$, the ball of center $x$ and radius $r$ is the set $B(x,r) = \set*{y:\; d(x,y) < r}$. If $B = B(x,r)$ is the ball, $\theta \in \R$, then $\theta B$ refers to the set $B(x, \theta r)$.

For $p \geq 1$, we let $\|\cdot\|_p$ be the $L^p$ norm on $(M, \mu)$. We define 
\begin{equation*}
	\|f\|_p = \left( \int_M |f|^p \diff\mu \right)^{1/p}.
\end{equation*} 

For $T$ a bounded operator on $L^p$, we use $\|T\|_{L^p \goesto L^p}$, or $\|T\|_p$ when there is no confusion, to refer to its operator norm:
\begin{equation*}
	\|T\|_p = \sup_{\substack{\psi \in L^p\\ \psi \neq 0}} \frac{\|T\psi\|_p}{\|\psi\|_p}.
\end{equation*}

For an open set $U \subset M$, $\lambda_1(U)$ refers to lower bound of the spectrum of $\Delta_\mu$ on $U$:

\begin{equation}
	\lambda_1(U) = \inf_{\substack{\psi\in\SmoothComp{U}\\ \psi \neq 0}} \frac{\|\nabla \psi\|^2_2}{\|\psi\|_2^2}.
\end{equation}

When $H$ is a symmetric operator defined on smooth functions with compact support, $\lambda_1(H)$ is similarly defined to be:

\begin{equation}
	\lambda_1(H) = \inf_{\substack{\psi\in\SmoothComp{M}\\ \psi \neq 0}} \frac{\llangle H\psi, \psi \rrangle}{\|\psi\|_2^2}.
\end{equation}

On a weighted manifold $(M,g,\mu)$, we define the Morrey norms $N_p$, $p\geq 0$, as follows:

\begin{equation}\label{eq:Morrey}
	\forall f \in L^1_{loc}(M),\;N_p(f) = \sup_{\substack{x \in M \\ r > 0}} \left( r^{2p} \fint_{B(x,r)} |f|^p \diff\mu \right)^{1/p},
\end{equation}

\noindent where $\fint_B f \diff\mu = \frac{1}{\mu(B)} \int_B f \diff\mu$ is the mean of $f$ over $B$. We also define the Morrey norm taken on balls of radius less than $R > 0$:

\begin{equation}\label{eq:def morrey norm}
	N_{p,R}(f) = \sup_{\substack{x \in M \\ 0 < r < R}} \left( r^{2p} \fint_{B(x,r)} |f|^p \diff\mu \right)^{1/p}.
\end{equation}

For our generalization to hold, it is important that $(M,g,\mu)$ must admits a \emph{relative Faber Krahn inequality} (property $\FK{}$) defined as follows:

\begin{definition}
	A weighted Riemannian manifold $(M, g, \mu)$ admits a relative Faber-Krahn inequality if there exist constants $b, \eta > 0$, such that for all $x \in M$, $r > 0$, and for any relatively compact open set $U \subset B(x,r)$, the following inequality holds:
	
	\begin{equation}\label{eq:FK}
		\lambda_1(U) \geq \frac{b}{r^2} \left( \frac{\Vol{x}{r}}{\mu(U)}\right)^\frac{2}{\eta}.
	\end{equation}
	
	We say that $M$ admits a relative Faber-Krahn inequality at scale $R$ (property $\FK{R}$) if \eqref{eq:FK} holds only for $0 \leq r \leq R$.
\end{definition}

In what follows, we refer to the constants $b$, $\eta$ in \eqref{eq:FK} as the \emph{Faber-Krahn constants} of the manifold.

\subsection{Statements of the results}

\begin{theorem}\label{th:Fefferman-Phong generalized}
	Let $(M, g, \mu)$ be a weighted complete Riemannian manifold satisfying $\FK{}$, then for any $p > 1$, there is a constant $C_p$ depending only on the Faber-Krahn constants and on $p$, such that for any $V\in L^1_{loc}(M)$, $V \geq 0$, and any $\psi \in \SmoothComp{M}$, the following inequality holds: 
	\begin{equation}\label{eq:Fefferman_Phong generalized}
		\int_M V\psi^2 \diff\mu \leq C_p N_{p}(V) \int_M |\nabla\psi|^2 \diff \mu.
	\end{equation}
\end{theorem}

If only $\FK{R}$ holds, then we can prove the following localized inequality:

\begin{theorem}\label{th:Weak_positivity}
	Let $(M,g,\mu)$ be a complete weighted Riemannian manifold, such that, for some $R > 0$, $\FK{R}$ holds. Then, for any $p > 1$, there is a constant $C_p > 0$ depending only on the Faber-Krahn constant and on $p$, such that for any $V\in L^1_{loc}(M), V \geq 0$, and any $\psi \in \SmoothComp{M}$, the following inequality holds:
	\begin{equation}\label{eq:Weak_majoration}
		\int_M V\psi^2 \diff\mu \leq C_p N_{p, R}(V) \left(\int_M |\nabla\psi|^2 \diff\mu + \frac{1}{R^2} \int_M \psi^2 \diff\mu \right).
	\end{equation}
\end{theorem}

From this inequality we can generalize the Fefferman Phong estimate on the lower bound of the spectrum of the operator $H = \Delta - V$. Indeed if $\FK{}$ holds, then for any $R > 0$, $\FK{R}$ is satisfied. Thus \eqref{eq:Weak_majoration} is true for any $R$. Then the following theorem follows easily:

\begin{theorem}\label{th:lower_bound estimates}
	Let $(M,g,\mu)$ be a complete weighted Riemannian manifold satisfying $\FK{}$. Then for any $p > 1$ there exist constants $C_1, C_p > 0$ depending only on the Faber-Krahn constants (and $C_p$ depending also on $p$), such that, for any $V \in L^1_{loc}(M)$, $V \geq 0$, and for the operator $H = \Delta_\mu - V$ the following inequalities hold:	
	\begin{equation}
		\sup_{\substack{x\in M\\ \delta > 0}} \left( C_1 \fint_{B(x,\delta)} V \diff \mu - \delta^{-2} \right) \leq -\lambda_1(H) \leq \sup_{\substack{x\in M\\\delta > 0}} \left(C_p \left( \fint_{B(x,\delta)} V^p \diff\mu \right)^{1/p} - \delta^2 \right).
	\end{equation}	
\end{theorem}

In addition, if $\lambda_1(M) > 0$, then we can strengthen \eqref{eq:Weak_majoration}, and obtain the following result, giving a condition for $\Delta - V$ to be positive:

\begin{theorem}\label{th:Positive lambda_1}
	Let $(M,g,\mu)$ be a complete weighted Riemannian manifold, such that $\FK{R}$ holds for $R > 0$. If in addition, $\lambda_1(M) > 0$, then for any $p > 1$, there is a constant $C_p > 0$ depending only on the Faber-Krahn constants such that, for $V\in L^1_{loc}(M), V \geq 0$, and any $\psi \in \SmoothComp{M}$, the following inequality holds:
	\begin{equation}\label{eq:Positive lambda_1}
		\int_M V\psi^2 \diff\mu \leq C_p N_{p, R}(V)\frac{1 + \lambda_1(M)R^2}{\lambda_1(M)R^2} \left(\int_M |\nabla\psi|^2 \diff\mu +	 \frac{\lambda_1(M)}{2}\int_M \psi^2 \diff\mu \right).
	\end{equation}
\end{theorem}

\subsection{$L^2$ Hardy inequality}

Notice that the inequality \eqref{eq:Fefferman_Phong generalized} is, for potentials $V$ with $N_p(V) < +\infty$, nothing more than the generalized $L^2$ Hardy inequality:

\begin{equation}\label{eq:Hardy_rho}
	\forall \psi \in \SmoothComp{M},\, \int_M \frac{\psi^2}{\rho^2} \diff\mu \leq C \int_M |\nabla \psi|^2 \diff\mu,
\end{equation}

with $\rho = V^{-1/2}$. Thus, on manifolds for which theorem \ref{th:Fefferman-Phong generalized} holds, the "classical" Hardy inequality, where $\rho$ is the distance to a point, is true whenever $N_p(d(o,\cdot)^{-2})$ is finite. For this to hold, we must make an additional assumption on the measure $\mu$.

\begin{definition}
	A metric measure space $(X,d,\mu)$ satisfies the reverse doubling property of order $\nu$, property $\rd{}$, (or, $\mu$ is $\nu$-reverse doubling) if, there is some constant $a > 0$ such that for all $x\in M$, $0 < r \leq r'$, the following inequality holds:	
	\begin{equation}
		a \parenfrac{r'}{r}^\nu \leq \frac{\Vol{x}{r'}}{\Vol{x}{r}}.
	\end{equation}
\end{definition}

\begin{theorem}\label{th:Hardy}
Let $(M,g,\mu)$ be a weighted Riemannian manifold. Assume that $M$ satisfies $\FK{}$, and that $\mu$ satisfies $\rd{}$ with $\nu > 2$. Then there is some constant $C > 0$ depending only on the Faber-Krahn and reverse doubling constants, such that, for any $o \in M$, then for any $\psi \in \SmoothComp{M}$ the following inequality holds:
	\begin{equation}\label{eq:Hardy}
		\int_M \frac{\psi(x)^2}{\rho(x)^2} \diff\mu(x) \leq C \int_M |\nabla \psi|^2 \diff\mu,
	\end{equation}
	
	\noindent with $\rho(x) = d(o,x)$.
\end{theorem}

We can compare this to the results of V. Minerbe \cite{Minerbe09} or G. Grillo \cite{Grillo03}, who proved $L^p$ Hardy inequalities assuming a Poincaré inequalities and a doubling measure. While we only get a $L^2$ inequality, it holds true under the weaker hypothesis of a relative Faber-Krahn inequality.

A recent work by Cao, Grigor'yan and Liu \cite{Cao2020HardysIA} proved Hardy inequalities as a consequence of volume doubling, reverse doubling, and certain estimates on either the Green function or the heat kernel. Their results are far more general than what we prove on Hardy inequality here.

\subsection{Examples}

We give various cases of manifolds which will satisfy a relative Faber-Krahn inequality (or a relative Faber-Krahn inequality at scale $R$). Then, theorem \ref{th:Fefferman-Phong generalized} (or theorem \ref{th:Weak_positivity}) holds.

\subsubsection{Complete manifolds with Ricci curvature bounded from below}

From  Li and Yau\cite{LiYau86}, the heat kernel of a complete manifold $(M,g,\mu)$ of dimension $n$, with $\mu$ here being the Riemannian volume measure, with Ricci curvature bounded from below by $-K$, for a constant $K \geq 0$, admits the following diagonal estimate:

\begin{equation*}
	p_t(x,x) \leq \frac{C_0}{\Vol{x}{\sqrt{t}}} e^{C_1 K t}.
\end{equation*}

Also, as a consequence of the Bishop-Gromov volume comparison theorem, we get that (see \cite{CheegerGromovTaylor82,ChavelBook06, SaloffCosteBook02} for example), for any $0 < r \leq r'$:

\begin{equation*}
	\frac{\Vol{x}{r'}}{\Vol{x}{r}} \leq \parenfrac{r'}{r}^n \exp\left(\sqrt{(n-1)K}r' \right).
\end{equation*}

Those two conditions implies, (see for example \cite{SaloffCosteBook02, HebischSaloff-Coste01}, or proposition \ref{pr:heat + doubling implies FKR} later), that there is some $R > 0$ such that $M$ satisfies $\FK[n]{R}$. If the Ricci curvature is non-negative, then we also have $\FK[n]{}$.

\subsubsection{Manifolds satisfying Faber Krahn inequalities outside a compact set}

We consider a complete weighted manifold $M$, and remove from it a compact set with smooth boundary $K$. We let $E_1, \ldots, E_k$ be the connected components of $M\setminus K$, and suppose that each $E_i$ is the exterior of a compact set with smooth boundary in a complete manifold $M_i$.

A simple example of such manifold is the connected sum of two (or more) copies of $\R^n$. It admits $\FK[n]{}$, but it is known that such manifold doesn't satisfy a Poincaré inequality (see for example \cite{BenjaminiItaiChavel96}).

Using \cite{GrigoryanSaloffCoste16}, we get that if each $M_i$ satisfies $\FK{}$, then there is some $R > 0$ such that $M$ satisfies $\FK{R}$.

\subsubsection*{Acknowledgements}

I thank G. Carron for his many advices and remark that helped shape this article into its present form, P. Castillon and L. Guillopé for their advices comments on the results and the manuscript. I also thank the Centre Henri Lebesgue \textbf{ANR-11-LABX-0020-01} for creating an attractive mathematical environment. I was partially supported by the ANR grant: \textbf{ANR-18-CE40-0012}: RAGE.

\section{Some techniques of harmonic analysis}\label{sec:Harmonic analysis}

\begin{remark}
The letter $c$, $C$ will usually be used for generic constants, which value might change from line to line. When the dependance on some parameter is judged important and non obvious, it will be made clear when it appears, before being folded into the generic constants on subsequent lines.
\end{remark}

\subsection{Dydadic cubes}

In $\R^n$, the natural decomposition of the space into cubes of length $2^k$, $k \in \Z$ is a very powerful tool. It turns out that families of open sets satisfying similar properties to those of the dyadic cubes in the euclidean space can be constructed in a more general setting.

We will use the construction of such "dyadic cubes" given by E. Sawyer and R. L. Wheeden in \cite{SawyerWheeden92} (though other such constructions, such as the one given in \cite{Christ90}, could also be used without major changes). Though it remains true in a more general setting, for our purposes it can be stated as:

\begin{theorem}\label{cubes}
	Let $(X, d)$ be a separable metric space, then there is a constant $\rho > 1$ ($\rho = 8$ works), such that for any (large negative) integer $m$, there are points $\set*{x_\alpha^k}$ and a family $\collection{D}_m = \set*{\mathcal{E}_\alpha^k}$ of Borel sets for $k = m,\, m+1,\,\ldots$, $\alpha = 1, 2, \ldots$, which satisfy the following properties:
	\begin{itemize}
		\item $B(x_\alpha^k, \rho^k) \subset \mathcal{E}_\alpha^k \subset B(x_\alpha^k, \rho^{k+1})$.
		\item For each $k = m, m+1, \ldots$, the family $\set*{\mathcal{E}_\alpha^k}_\alpha$ is pairwise disjoint in $\alpha$ and $X = \bigcup_\alpha \mathcal{E}_\alpha^k$.
		\item If $m \leq k < l$, then either $\mathcal{E}_\alpha^k \cap \mathcal{E}_\beta^l = \emptyset$ or $\mathcal{E}_\alpha^k \subset \mathcal{E}_\beta^l$.
	\end{itemize}
\end{theorem}	
	
Given such a family $\collection{D}_m$, the sets $\mathcal{E}_\alpha^k$ will be called \emph{dyadic cubes} of $M$, or simply \emph{cubes}. The ball $B(x_\alpha^k, \rho^{k+1})$  is called the containing ball of the cube $\mathcal{E}_\alpha^k$. For any cube $Q$ the containing ball is denoted by $B(Q)$. $\rho$ will be called the sidelength constant of dyadic cubes.
	
The length of a cube $Q$ is the radius of $\rho^{-1}B(Q)$, written $\ell(Q)$.

\subsection{Properties of doubling measures}

We start by recalling the definitions and some standard properties of doubling measures. Most of the proofs are classical, but are rarely explicitely done for the $R$ doubling case, and we thus give them for completeness' sake, without claiming originality.

\begin{definition}\label{def:def doubling}
	A metric measure space $(X,d,\mu)$ satisfies the doubling property $\doubling{}$ of order $\eta$ if, there is some constant $A > 0$ such that for all $x\in M$, $0 < r \leq r'$, the following inequality holds:	
	\begin{equation}\label{eq:def doubling}
		\frac{\Vol{x}{r'}}{\Vol{x}{r}} \leq A \parenfrac{r'}{r}^\eta.
	\end{equation}
	
	We call $A$ the doubling constant, and $\eta$ the doubling order. We will also say "the doubling constants" to refer to both $A$ and $\eta$ at the same time. The property $\doubling{}$ is equivalent to the fact that for some constant $A > 0$, for any ball $B \subset M$:
	
	\begin{equation}\label{eq:def doubling alt}
		\mu(2B) \leq A \mu(B).
	\end{equation}
\end{definition}

The proof of the equivalence is the same as that of the $R$-doubling case given after definition \ref{def:R doubling}, (with $R = \infty$).

A note on the constants: \eqref{eq:def doubling alt} implies \eqref{eq:def doubling} with $\eta = \log_2 A$ (and $A$ the same in both inequalities), while conversely, \eqref{eq:def doubling} implies that the constant in \eqref{eq:def doubling alt} be $2^\eta A$. By increasing $A$ and $\eta$ if necessary, we can always assume that $A = 2^\eta$.

We repeat, for completeness, the definition of the reverse doubling property:

\begin{definition}\label{def:def reverse doubling}
	A metric measure space $(X,d,\mu)$ satisfies the reverse doubling property $\rd{}$ of order $\nu$ if, there is some constant $a > 0$ such that for all $x\in M$, $0 < r \leq r'$, the following inequality holds:	
	\begin{equation}\label{eq:def reverse doubling}
		a \parenfrac{r'}{r}^\nu \leq \frac{\Vol{x}{r'}}{\Vol{x}{r}}.
	\end{equation}
	
	We call $a$ the reverse doubling constant, and $\nu$ the reverse doubling order. The property $\rd{}$ is equivalent to the fact that for some constant $a \in (0,1)$, for any ball $B \subset M$:
	
	\begin{equation}\label{eq:def reverse doubling alt}
		\mu(B) \leq a\mu(2B).
	\end{equation}
\end{definition}

\begin{proof}[Proof of \eqref{eq:def reverse doubling alt} implies \eqref{eq:def reverse doubling}]
	We can assume that $a \leq 1$. Let $x\in X$, $0 < r \leq r'$. Writing $\lfloor t \rfloor$ for the integer part of $t\in \R$, let $k = \left\lfloor \log_2 \frac{r'}{r}\right\rfloor$. Then: 
	
	\begin{align*}
		\Vol{x}{r} &\leq a^k\Vol{x}{2^kr} \\
		&\leq a^k \Vol{x}{r'} \\
		&\leq a^{-1 + \log_2 \frac{r'}{r}} \Vol{x}{r'}\quad (a \leq 1) \\
		&\leq \frac{1}{a} \parenfrac{r'}{r}^{-\nu} \Vol{x}{r'},
	\end{align*}
	
	with $\nu = -\log_2 a$. Thus:
	
	\begin{equation*}
		a \parenfrac{r'}{r}^\nu \leq \frac{\Vol{x}{r'}}{\Vol{x}{r}}.
	\end{equation*}
\end{proof}

\begin{proposition}\label{pr:doubling different centers}
	Let $(X, d, \mu)$ satisfies $\doubling{}$, then for any $x,y \in M$, $r, r' > 0$ such that $B(y, r) \subset B(x,r')$, we have:
	
	\begin{equation}\label{eq:doubling different centers}
		\frac{\Vol{x}{r'}}{\Vol{y}{r}} \leq A^2 \parenfrac{r'}{r}^\eta.
	\end{equation}
\end{proposition}

	This is a classical result. The proof is similar to what we will do to prove proposition \ref{pr:R doubling different centers}.

\begin{definition}\label{def:R doubling}
	A metric measure space $(X,d,\mu)$ satisfies the $R$-doubling property $\doubling{R}$ if there is some constant $A > 0$ such that \eqref{eq:def doubling} holds for all $x \in M$, and $0 < r \leq r' \leq 2R$. This is equivalent to \eqref{eq:def doubling alt} being true for all ball $B$ with radius less than $R$.
	
	$X$ satisfies the $R$-reverse doubling property $\rd{R}$ if \eqref{eq:def reverse doubling alt} holds for all balls of radius less than $R$ (this is equivalent to \eqref{eq:def reverse doubling} being true for all $x \in X$ and $0 < r \leq r' \leq 2R$).
\end{definition}

We will write $A_R$ for the doubling constant when it's important to precise which $R$ the constant is associated with.

Some care is needed to get precisely those maximal radius. That \eqref{eq:def doubling alt} follows from \eqref{eq:def doubling} is immediate.

\begin{proof}[Proof of \eqref{eq:def doubling alt} implies \eqref{eq:def doubling}]
	Suppose that there is some constant $A$ such that for all ball $B$ of radius less than $R$, then $\mu(2B) \leq A\mu(B)$. Let $r \leq r' \leq 2R$, $k = \left\lfloor \log_2 \frac{r'}{r} \right\rfloor$. 
	
	We have:
	\begin{equation*}
		2^{-k-1}r' < r \leq 2^{-k}r',
	\end{equation*}
	\noindent and, using repeatedly the doubling inequality $\Vol{x}{\rho} \leq A \Vol{x}{\rho/2}$, valid for all $\rho \leq 2R$, we have:	
	\begin{align*}
		\Vol{x}{r'} &\leq A^{k+1}\Vol{x}{2^{-k - 1} r'} \\
		&\leq A^{k + 1} \Vol{x}{r} \\
		&\leq A e^{\left(\log A \log \frac{r'}{r}\right)/ \log 2} \Vol{x}{r} \\
		&\leq A \left(\frac{r'}{r}\right)^\eta \Vol{x}{r},
	\end{align*}
	
	with $\eta = \log_2 A$.
\end{proof}

\begin{proposition}\label{pr:R doubling different centers}
	Let $X$ satisfies $\doubling{R}$, then for all $x, y \in X$, $r, r' > 0$ such that $B(y, r) \subset B(x,r')$ and with $r' < R$, then for $\eta = \log_2 A$:
	
	\begin{equation}\label{eq:R doubling different centers}
		\frac{\Vol{x}{r'}}{\Vol{y}{r}} \leq A^2 \left(\frac{r'}{r}\right)^\eta.
	\end{equation}
	
	If in addition $X$ satisfies $\rd{R}$, then we also have for some constant $c > 0$, that for all $0 < r, r' < R$ and $B(y,r) \subset B(x,r')$,
	
	\begin{equation}\label{eq:reverse doubling different centers}
		c \parenfrac{r'}{r}^\nu \leq \frac{\Vol{x}{r'}}{\Vol{y}{r}}.
	\end{equation}
\end{proposition}

\begin{proof}For the first part, we simply use $B(x,r) \subset B(y, 2r)$ then applies \eqref{eq:def doubling}.

For the second part, since $B(x,r') \subset B(y,2r')$, we can use \eqref{eq:doubling different centers} and we get:

\begin{align*}
	\frac{\Vol{x}{r'}}{\Vol{y}{r}} &= \frac{\Vol{y}{r'}}{\Vol{y}{r}}\frac{\Vol{x}{r'}}{\Vol{y}{r'}} \\
	&\geq a\parenfrac{r'}{r}^\nu\frac{\Vol{x}{r'}}{\Vol{y}{2r'}} \\
	&\geq a A^{-2} 2^{-\eta} \parenfrac{r'}{r}^\nu.
\end{align*}
\end{proof}

We now suppose that $(X, d)$ is a \emph{path metric space}, i.e.\ that the distance $d(x,y)$ is realised as the infimum of the length of continuous path with end points $x$ and $y$. We will keep making this assumption in everything that follows (Most results are still true in a more general setting, but this simplify some proofs and is sufficient for our purposes).

\begin{proposition}
	Let $X$ be a metric space satisfying $\doubling{R}$. Assume that the annuli $B(x,r') \setminus B(x,r)$, for any $r, r'$ with $0 \leq r < r' \leq R$ are all non empty. Then there is some $\nu > 0$ such that $X$ satisfies $\rd{R/2}$.
\end{proposition}

\begin{proof}
	Let $x \in X$, $r < R/2$. Take $y \in B(x,7r/4) \setminus B(x,5r/4)$ (which is non empty as $7r/4 \leq R$). Then: 
	\begin{equation*}
		B(y, r/4) \subset B(x,2r) \setminus B(x,r).
	\end{equation*}
	
	Then we have: 
	\begin{align*}
		\Vol{x}{2r} &\leq A^2 8^\eta \Vol{y}{r/4} \\
		\Vol{y}{r/4} &\leq \Vol{x}{2r} - \Vol{x}{r}
	\end{align*} 
	
	Thus with $C = A^2 8^\eta$, we have:
	\begin{equation*}
		(1 + C^{-1})\mu(B(x,r)) \leq \mu(B(x,2r)).
	\end{equation*}
	
	Thus the measure satisfies the $R$-reverse doubling property.
\end{proof}

The $R$-doubling also implies some upper bound on the volume of balls of large radius. The two following propositions, and their proof, are taken from \cite{HebischSaloff-Coste01}.

\begin{proposition}\label{pr:annuli_bound}
	If $(X,d, \mu)$ is a path metric space satisfying $\doubling{R}$, then there is some $C > 0$ that depends only on the doubling constant and order, such that we have, for any $r > 0$, $R' \leq R$:
	
	\begin{equation}
		\Vol{x}{r + R'/4} \leq C \Vol{x}{r}.
	\end{equation}
\end{proposition}

\begin{proof}
	The case $r \leq R$ is obvious by the doubling property. For $r > R$, then let $\set*{x_i}_i$ be a maximal family in $B(x, r - R/4)$ such that for any $i \neq j$, $d(x_i, x_j) > R'/2$. Then the balls $B(x_i, R'/4) \subset B(x,r)$ are disjoints, and the balls $B(x_i, R')$ cover $B(x,r+R'/4)$, since a point of $B(x,r+R'/4)$ is at distance at most $R'/2$ of $B(x,r - R'/4)$ (this because $(X,d)$ is a path-metric space). Thus
	
	\begin{equation*}\Vol{x}{r+R'/4} \leq \sum_i \Vol{x_i}{R'} \leq A^2 \sum_i \Vol{x_i}{R'/4} \leq A^2 \Vol{x}{r}.
	\end{equation*}
\end{proof}

\begin{proposition}\label{pr:exp doubling}
	If $(X,d, \mu)$ satisfies $\doubling{R}$ then, there is a $D > 0$, that depends only on the the doubling constants, such that for any $r > 0$, we have:
	
	\begin{equation}\label{eq:quasi_doubling}
		\Vol{x}{r} \leq e^{D\frac{r}{R}} \mu(B(x,R)). 
	\end{equation}
\end{proposition}

\begin{proof}
	Let $r > R$, $k = \left\lfloor 4\frac{r - R}{R}\right\rfloor$, then we have:
	\begin{equation*}
		\Vol{x}{r} \leq \Vol{x}{R + (k+1)R/4},
	\end{equation*} 
	
	\noindent thus by proposition \ref{pr:annuli_bound}, $\Vol{x}{r} \leq C^{k+1}\Vol{x}{R}$. Moreover, $k + 1 \leq 4\frac{r}{R} - 3 \leq 4\frac{r}{R}$, and so:
	
	\begin{equation*}
		\Vol{x}{r} \leq \exp\left(4 \ln\left(C\right) \frac{r}{R} \right) \Vol{x}{R},
	\end{equation*}
	\noindent and so we get \eqref{eq:quasi_doubling} with $D = 4 \ln(C)$.
	
	If $r \leq R$, then:
	\begin{equation*}
		\mu(B(x,r)) \leq \mu(B(x,R)) \leq e^{D\frac{r}{R}} \mu(B(x,R))
	\end{equation*}
	
	\noindent and thus \eqref{eq:quasi_doubling} still holds.
\end{proof}

Similarly to how we always use $A$ for the doubling constant, $D$ will always be used for this constant $D = 8 \log A$.

\begin{proposition}\label{pr:volume different centers}
	Let $X$ satisfies $\doubling{R}$, let $r \leq R$, then there exists a constant $C > 0$, that depends only on the doubling constant and order, such that for any $x, y \in X$, $\Vol{x}{r} \leq Ce^{D\frac{d(x,y)}{r}}\Vol{y}{r}$.
\end{proposition}

\begin{proof}
	We have the inclusion $B(x,r) \subset B(y, r + d(x,y)) \subset B(y, R + d(x,y))$. Then, by proposition \ref{pr:annuli_bound}, we have:
	\begin{equation*}
		\Vol{x}{r} \leq A^8 \Vol{y}{d(x,y)},
	\end{equation*}
	
	\noindent then using proposition \ref{pr:exp doubling}:	
	\begin{equation*}
		\Vol{x}{r} \leq C e^{D\frac{d(x,y)}{R}} \Vol{y}{r} \leq C e^{D\frac{d(x,y)}{r}}\Vol{y}{r}.
	\end{equation*}
\end{proof}

\begin{proposition}\label{pr:bigger R}
	If $(X,d, \mu)$ satisfies $\doubling{R}$, then it also satisfies $\doubling{R'}$ for any $R' > 0$, with a doubling constant $A_{R'} = A_R$ if $R' \leq R$, and $A_{R'} = e^{2D\frac{R'}{R}}$ if $R' > R$.
\end{proposition}

\begin{proof}
	The case $R' \leq R$ is obvious. Thus assume $R > R'$, let $r \leq R'$. If $r \leq R$ then the result is trivial since $A_{R} \leq A_{R'}$. If $r > R$, then by proposition \ref{pr:exp doubling}:
	\begin{equation*}
		\Vol{x}{2r} \leq e^{2D\frac{r}{R}} \Vol{x}{r}
	\end{equation*}
	
	Since $e^{2D\frac{r}{R}} \leq e^{2D\frac{R'}{R}}$, we conclude that $\mu$ is R'-doubling, with a doubling constant $A_{R'} = e^{2D\frac{R'}{R}}$.
\end{proof}

With this we can generalise proposition \ref{pr:volume different centers} for any $r > 0$: if $r > R$, we can use the $r$-doubling and apply proposition \ref{pr:volume different centers} for it. The constants are $A_r = e^{2D\frac{r}{R}}$, $D_r = 4\log \left(A_r^2\right) = 16 D \frac{r}{R}$, $A_r^8 = e^{16 D \frac{r}{R}}$. Then we have, for any $x,y\in X$, $r > 0$:

\begin{equation}
	\Vol{x}{r} \leq e^{16 D \frac{r + d(x,y)}{R}} \Vol{y}{r}.
\end{equation}

\begin{proposition}\label{pr:reverse_doubling bigger R}
	Let $(X,d,\mu)$ be a metric measure space that satisfies $\doubling{R}$. If it also satisfies $\rd{R}$, then for any $\kappa > 1$, it satisfies $\rd{\kappa R}$ with a different reverse doubling constant, that depends only on the doubling and reverse doubling constants, and on $\kappa$.
\end{proposition}

The notable part of this proposition is that the reverse doubling order is the same.

\begin{proof}
	By proposition \ref{pr:bigger R}, $\mu$ is $\kappa R$-doubling for all $\kappa$, with some doubling order $\eta = \eta(\kappa)$. We take a point $x\in M$, and $r, r'$ with $0 < r \leq r' \leq \kappa R$. We want to prove that there's some constant $a_\kappa$ such that, for any such $x, r, r'$:
	
	\begin{equation*}\label{eq:reverse_kappa_doubling}
		\frac{\Vol{x}{r'}}{\Vol{x}{r}} \geq a_\kappa \left(\frac{r'}{r}\right)^\nu.
	\end{equation*}
	
	If $0 < r \leq r' \leq R$, then there's nothing to do but apply $\rd{R}$. If $0 < r \leq R < r' \leq \kappa R$, then:
	
	\begin{equation*}
		\frac{\Vol{x}{r'}}{\Vol{x}{r}} \geq \frac{\Vol{x}{R}}{\Vol{x}{r}} \geq a \left(\frac{R}{r}\right)^\nu \geq a \kappa^{-\nu} \left(\frac{r'}{r}\right)^\nu.
	\end{equation*}
	
	Finally, when $R < r \leq r' < \kappa R$, then:
	
	\begin{align*}
		\frac{\Vol{x}{r'}}{\Vol{x}{r}} 
		&\geq \frac{\Vol{x}{\frac{r'}{\kappa}}}{A \kappa^\eta \Vol{x}{\frac{r}{\kappa}}} \\
		&\geq \frac{a}{A\kappa^\eta} \parenfrac{r'}{r}^\nu
	\end{align*}
	
	Thus \eqref{eq:reverse_kappa_doubling} holds for $a_\kappa= \min\left(a, a\kappa^{-\nu}, aA^{-1}\kappa^{-\eta}\right) = a A^{-1}\kappa^{-\eta}$.
\end{proof}

\begin{proposition}\label{p_covering_card}
	Let $(X,d,\mu)$ satisfies $\doubling{R}$. Take $x \in X$, $r > 0$, and let $B = B(x,r)$. Let $\delta$ be such that $0 < \delta \leq \min(r, R)$, and $\set{x_i}_i \subset B$ be a family of points such that the balls $B_i = B(x_i, \delta)$ form a covering of $B$ and that for any $i \neq j$, $\frac{1}{2}B_i \cap \frac{1}{2} B_j = \emptyset$.

	Then there are constants $C, c$, depending only on the doubling constant such that
	
	\begin{equation}
		\card(I) \leq C e^{c\frac{r}{\delta}}.
	\end{equation}
\end{proposition}
\begin{proof}
	For any $i$, $B_i \subset B(x, r+ \delta)$, and since $\delta \leq R$, then we can use proposition \ref{pr:annuli_bound} to get
	\begin{equation*}
		\Vol{x}{r + \delta} \leq C \Vol{x}{r}.
	\end{equation*}
	
	Now, if $r > R$, then by proposition \ref{pr:exp doubling}, since $\delta \leq R$ then $\mu$ is $\delta$ doubling with the same doubling constant as that of the $R$-doubling, and:
	 \begin{equation*}
	 	\Vol{x}{r} \leq e^{D\frac{r}{\delta}}\Vol{x}{\delta}
	 \end{equation*}
	
	Moreover by proposition \ref{pr:volume different centers}: 
	\begin{equation*}
		\Vol{x}{\delta} \leq C e^{D\frac{d(x,x_i)}{\delta}}\Vol{x_i}{\delta} \leq C e^{D\frac{r}{\delta}}\mu(B_i),
	\end{equation*}
	
	\noindent using that, since $x_i \in B$, then $d(x,x_i) \leq r$. Thus we have $\mu(B(x,r)) \leq C e^{2D\frac{r}{\delta}} \mu(B_i)$, and the constant $C$ depends only on the doubling constants. We then have:	
	\begin{align*}
		\left(\card{I} \right)\Vol{x}{r + \delta} &\leq Ce^{2D\frac{r}{\delta}}  \sum_{i \in I} \mu(B_i) \\
		&\leq ACe^{2D\frac{r}{\delta}}  \sum_i \mu\left(\frac{1}{2}B_i\right) \\
		&\leq Ce^{2D\frac{r}{\delta}}  \Vol{x}{r+ \delta}.
	\end{align*}
	
	Thus $\card(I) \leq Ce^{2D\frac{r}{\delta}} $ and the constant $C$ depends only on the doubling constants.
\end{proof}

\begin{remark}
	For any ball $B$, such a covering always exists: take for $\set{x_i}_i \subset B$ a maximal family with $d(x_i, x_j) \geq \delta$ for any $i\neq j$.
\end{remark}

\begin{proposition}
	Let $M_R$ be the centered maximal function defined by:
	
	\begin{equation}
		\forall f\in L_{loc}^1(M),\; M_R f(x) = \sup_{r < R} \fint_{B(x,r)} |f| \diff \mu.
	\end{equation}
	
	Then, if $\mu$ satisfies $\doubling{R}$, $M_{R/2}$ is bounded on $L^p$ for all $p \in (1, +\infty]$, and the operator norm is bounded by a constant that only depends on the doubling constant $A$ and on $p$.
\end{proposition}

We will use the following classical results:

\begin{lemma}[Vitali's covering lemma]\label{Vitali lemma}
	Let $(X,d)$ be a separable metric space, and $\set{B_j}_{j\in J}$ a collection of balls, such that $\sup_j r(B_j) < \infty$. Then for any $c > 3$ there exists a subcollection $\set{B_{j_n}}_{n\in \N} \subset \set{B_j}_{j\in J}$ such that the $B_{j_n}$ are pairwise disjoint and $\bigcup_{j\in J} B_j \subset \bigcup_{n\in \N} c B_{j_n}$.
\end{lemma}

\begin{theorem}[Marcinkiewicz interpolation theorem]\label{th:Marcinkiewicz}
	Let $(X,\mu)$ be a measure space, $T$ a sublinear operator acting on functions, i.e.\ there is a $\kappa > 0$ such that for any $f, g$ measurable, then $Tf, Tg$ are measurable and $T(f+g)(x) \leq \kappa \left(Tf(x) + Tg(x)\right)$ for almost every $x \in X$.
	
	Let $1 \leq p < r \leq \infty$. If $r < \infty$, assume that:	
	\begin{align*}
		\forall f \in L^p,\; \mu\set*{x\in X:\; Tf(x) > \lambda} &\leq \frac{A}{\lambda^p}\|f\|_p^p,\\
		\forall f \in L^r,\; \mu\set*{x\in X:\; Tf(x) > \lambda} &\leq \frac{B}{\lambda^r}\|f\|_r^r,
	\end{align*}
	
	If $r = \infty$, then assume instead that: 
	\begin{align*}
		\forall f \in L^p,&\; \mu\set*{x\in X:\; Tf(x) > \lambda}\leq \frac{A}{\lambda^p}\|f\|_p^p,\\
		\forall f \in L^\infty,&\; |Tf(x)|\leq B|f(x)|,\quad {a.e.\ } x \in X
	\end{align*}
	
	Then, for every $s \in (p,r)$, for all $f \in L^s$, $Tf\in L^s$ and:
	
	\begin{equation}
		\|Tf\|_s \leq C(A,B, p, r, s, \kappa) \|f\|_s.
	\end{equation}
\end{theorem}

\begin{proof}[Proof of the proposition]
	We have, for any $f \in L^\infty(M)$, $\|M_R f\|_\infty \leq \|f\|_\infty$.
	
	If $f \in L^1(M)$, then for any $\lambda > 0$, define:
	\begin{equation*}
		E_\lambda = \set*{x\in M:\; M_{R/2}f(x) > \lambda}.
	\end{equation*}
	
	If $x \in E_\lambda$, then there is some $r_x > 0$ such that $\lambda < \fint_{B(x,r_x)} |f| \diff \mu$, and $2 r_x \leq R$. Then:
	\begin{equation*}
		\mu(B(x,r_x)) \leq \lambda^{-1} \int_{B(x,r)} |f| \diff\mu.
	\end{equation*}
	
	We have $E_\lambda \subset \bigcup_x B(x,r_x)$, thus by Vitali's covering lemma, there is a subcollection $\set*{x_n}$ such that the $B(x_n, r_n)$ are pairwise disjoint and $E_\lambda \subset \bigcup_n B(x_n, 4 r_n)$.
	
	Also, since $r_n < R/2$, and $\mu$ is R-doubling, we have $\Vol{x_n}{4r_n} \leq A^2 \Vol{x_n}{r_n}$. Then:	
	\begin{align*}
		\mu(E_\lambda) &\leq \sum_n \mu(B(x_n, 4r_n)) \\
		&\leq A^2 \sum_n \mu(B(x_n, r_n)) \\
		&\leq A^2 \lambda^{-1} \sum_n \int_{B(x_n, r_n)} |f| \diff\mu  \\
		&\leq A^2 \frac{\|f\|_1}{\lambda}.
	\end{align*}
	
	So, by the Marcinkiewicz interpolation theorem, for any $p \in (1, +\infty)$, $M_{R/2}$ is bounded on $L^p$ with an operator norm $\|M_{R/2}\|_{p\goesto p} \leq C_p$, with $C_p$ depending only on $A$ and $p$.
\end{proof}

\begin{remark}
	Of course, $\doubling{R}$ implies $\doubling{R'}$ for all $R' > R$, then $M_R$ itself is also bounded, but with the constant $C_p$ depending on the constant for $\doubling{2R}$. And so are all the $M_{R'}$ with $R' > R$, with the constant $C_p$ depending on $p$, the $R$-doubling constant, and the ratio $R'/R$.
\end{remark}

\begin{proposition}\label{pr:uncentered equiv centered}
	Let $\tilde{M}_R$ the uncentered maximal function defined by: for all $f\in L_{loc}^1(M)$, 
	
	\begin{equation}
		\tilde{M}_Rf(x) = \sup_{\substack{x \in B,\\r(B) \leq R}} \fint_B |f|\diff\mu.
	\end{equation}
	
	With this supremum to be interpretated as being over all balls $B$ satisfying the given condition, and $r(B)$ being the radius of $B$.
	
	Then, if $\mu$ is R-doubling, there exist some constant $C > 0$ such that $M_R \leq \tilde{M}_R \leq C M_{2R}$.
\end{proposition}

\begin{proof}
	Since a ball centered at $x$ is a ball containing $x$, $M_R \leq \tilde{M}_R$ is obvious. Now, for some balls $B= B(y,r)$ containing $x$, with radius less than $R$, we have $B \subset B(x, 2r)$ and:
	
	\begin{equation*}
		\fint_B |f| \diff \mu \leq \frac{\Vol{x}{2r}}{\mu(B)}\fint_{B(x,2r)} |f| \diff \mu \leq C M_{2R} f(x).
	\end{equation*}
\end{proof}

\begin{proposition}\label{p_dyad_max}
	Let $(X,d,\mu)$ be a separable metric measure space, and $\dyadic_m$ be a chosen construction of dyadic cubes on $X$. Define the associated dyadic maximal function $M_{d,m}$ by:
	
	\begin{equation}
		M_{d,m}f(x) = \sup_{\substack{Q \in \collection{D}_m \\ x \in Q}} \fint_Q |f| \diff \mu.
	\end{equation}
	
	Then there is a constant $C_p$ such that for any $p > 1$, for any $f\in L^p$, $\|M_{d,m}f\|_p \leq C_p \|f\|_p$.
	
	As a consequence, $M_{d,m,l}$, the maximal function defined the same way, but with the cubes in the supremum being only those of length less than $l$, is also bounded on $L^p$ for all $p > 1$.
\end{proposition}

\begin{proof}
	Let $f \in L^1(X)$, $\lambda > 0$, we define: 
	\begin{equation*}
		E_\lambda = \set*{x\in X:\; M_{d,m}f(x) > \lambda}.
	\end{equation*}
	
	If $x \in E_\lambda$, then there is a cube $Q\in \dyadic_m$ such that $\fint_Q |f| \diff\mu > \lambda$, and so $Q \subset E_\lambda$. Then there are two possibilities: 
	
	If there is a maximal dyadic cube $P$ containing $x$ such that $\fint_P |f| \diff\mu > \lambda$. This cube satisfies $P \subset E_\lambda$.
	
	If there is no such cube (in which case, $x$ is in a region of space with infinite diameter but finite measure), then define $\Omega = \bigcup_{\substack{Q\in \dyadic_m \\ x\in Q}} Q$. We can always find an arbitrarily large cube containing $x$ which is a subset of $E_\lambda$, and so $\Omega \subset E_\lambda$, and $\mu(\Omega) \leq \lambda^{-1} \int_\Omega |f|\diff\mu < \infty$.
	
	Then take $\set{Q_i}_i$ to be the family of all the maximal dyadic cubes such that $\fint_{Q_i} |f|\diff\mu > \lambda$, and $\set{\Omega_j}_j$ be the family of all the the regions $\Omega_j = \bigcup_k Q_k^j$, where $\set{Q_k^j}$ is an infinite increasing sequence of cubes	with $\fint_{Q_j^k} |f|\diff\mu > \lambda$. The $Q_i, \Omega_j$ are pairwise disjoints: first it is clear by maximality that the $Q_i$ are. Then, if for a cube $Q$, we have $Q \cap \Omega_j \neq \emptyset$, then there is a cube $P\subset \Omega_j$ such that $P \cap Q \neq \emptyset$, thus we have either $P\subset Q$ or $Q \subset P$. In both case, $Q\subset \Omega_j$ since $\Omega_j$ is the union of all cubes containing $P$. This mean both that $Q_i \cap \Omega_j = \emptyset$ for all $i,j$, and that $\Omega_j \cap \Omega_l = \emptyset$ for $j\neq l$.
	
	Thus, we have the disjoint union:
	
	\begin{equation*}
		E_\lambda = \bigcup_i Q_i \cup \bigcup_j \Omega_j,
	\end{equation*}
	
	Then $\mu(Q_i) < \lambda^{-1}\int_{Q_i} |f| \diff\mu$, and $\mu(\Omega_j) \leq \lambda^{-1}\int_{\Omega_j}|f|\diff\mu$. Summing on all cubes and all regions, $\mu(E_\lambda) \leq \lambda^{-1}\int_{E_\lambda} |f| \diff\mu \leq \lambda^{-1} \|f\|_1$. Thus:
	
	\begin{equation}\label{eq:dyadic_max weak type}
		\mu\left( \set*{x\in X:\; M_{d,m}f(x) > \lambda}\right) \leq \frac{\|f\|_1}{\lambda}.
	\end{equation}
	
	Moreover, for $f\in L^\infty(X)$, we clearly have $M_{d,m}f(x) \leq \|f\|_\infty$. Then by Marcienkiewicz interpolation theorem, for any $p > 1$ there is a constant $C_p > 1$ such that $\|M_{d,m}f\|_p \leq C_p \|f\|_p$.
\end{proof}

\subsection{Estimates of operator norms by that of a maximal function}

We refers to the works of C. Pérez and R.L. Wheeden \cite{PerezWheeden03} for a more general approach. Well will first describe one of their result in the more specific context that is of interest to us, then will give a generalization of this result that holds on a $R$-doubling space.

In what follows, we let $(X,d)$ be a separable R-doubling metric space. We take $T$ an operator given by a kernel $K:X\times X\setminus \mathrm{Diag} \goesto \R$, i.e.\ 

\begin{equation}\label{T_def}
	Tf(x) = \int_X f(y)K(x,y)\diff\mu(y).
\end{equation}

We say that the operator $T$, or its kernel $K$, satisfies the condition $\Kcond{}$ if $K$ is non negative and if there are constants $C_1, C_2 > 1$ such that:

\begin{equation}\label{K_cond}
\begin{aligned}
	d(x', y) \leq C_2 d(x,y) &\Rightarrow K(x,y) \leq C_1 K(x', y), \\
	d(x, y') \leq C_2 d(x,y) &\Rightarrow K(x,y) \leq C_1 K(x, y').
\end{aligned}
\end{equation}

We take $\rho > 1$ such as, by theorem \ref{cubes}, for any integer $m\in \Z$, we have a decomposition of $X$ in dyadic cubes $\collection{D}_m$ of lenghts $\rho^\ell$, $\ell \geq m$. We define $\varphi$ as the following functional on balls 

\begin{equation}
	\varphi(B) = \sup_{\substack{x,y \in B\\ d(x,y) \geq \frac{1}{2\rho}r(B)}} K(x,y),
\end{equation}

and $M_\varphi$ to be the following maximal functions:

\begin{equation}
	M_\varphi f(x) = \sup_{\substack{x \in B}} \varphi(B) \int_{B} |f|\diff\mu.
\end{equation}

We want to establish an inequality of the type $\|Tf\|_p \leq C_p \|M_\varphi f\|_p$, as the later can be more convenient to estimate.

For $T$ satisfying $\Kcond{}$, it is shown in $(4.3)$ of \cite{SawyerWheedenZhao96} that $\varphi$ is decreasing in the following sense:

\begin{proposition}\label{p_phi_larger_ball}
	There is a constant $\alpha$, depending only on $C_1$, $C_2$, $\rho$ such that for any balls $B \subset B'$, $\varphi(B') \leq \alpha \varphi(B)$
\end{proposition}

\begin{proof}
	First we want to prove that if $\eqref{K_cond}$ holds, then, for any $C_2 > 1$, there exist a corresponding $C_1$ such that $\eqref{K_cond}$ holds with those new constants. We can of course replace $C_2$ by a smaller constant. To replace it with a smaller, we show that for any integer $k \geq 1$:
	\begin{equation*}
		d(x', y) \leq C_2^k d(x,y) \Rightarrow K(x,y) \leq C_1^k K(x',y),
	\end{equation*}
	\noindent and that the same holds with $(x,y')$ replacing $(x',y)$.
	
	We proceed by induction. The case $k = 1$ is simply \eqref{K_cond}. 
	
	Let $k > 2$, take $x, x', y \in X$ such that $d(x', y) \leq C_2^k d(x,y)$, and suppose that:
	\begin{equation*}
		d(x',y) \leq C_2^{k-1} d(x,y) \Rightarrow K(x,y) \leq C_1^{k-1}K(x'y),
	\end{equation*}
	
	\noindent then, if $d(x',y) \leq C_2^{k-1}d(x,y)$, the result holds and there is nothing to prove. If $d(x',y) > C_2^{k-1} d(x,y)$, then $X$ is a path metric space, so there is a path from $y$ to $x'$ of length $d(x',y)$, and on this path is a point $z$ such that $d(y,z) = C_2^{k-1} d(x,y)$. But then:
	\begin{equation*}
		d(x',y) \leq C_2^k d(x,y) = C_2 d(z,y),
	\end{equation*}
	
	\noindent thus $K(z,y) \leq C_1 K(x',y)$. 
	
	Then by induction, we proved that $K(x,y) \leq C_1^k K(x',y)$ for all $x, x', y$ with $d(x',y) \leq C_2^k d(x,y)$. It follows that if \eqref{K_cond} holds, then for any $C_2 > 1$ there exist a $C_1 > 1$ such that \eqref{K_cond} holds.
	
	Now we can prove the proposition proper. Take $x', y' \in B'$, $x,y \in B$ such that: 
	\begin{equation*}
		d(x',y') \geq c r(B'),\qquad d(x,y) \geq c r(B),
	\end{equation*}
	
	\noindent with $c = \frac{1}{2\rho}$. By exchanging $x'$ and $y'$ if necessary, we can suppose that $d(x,y') \geq d(x,x')$, then:
	\begin{equation*}
		cr(B') \leq d(x',y') \leq d(x', x) + d(x, y') \leq 2d(x,y').
	\end{equation*}
	
	Moreover, since $B \subset B'$, we have $d(x,y') \leq 2 r(B')$, and thus:	
	\begin{equation*}
		d(x,y') \leq \frac{2}{c} d(x',y'),
	\end{equation*}
	
	\noindent Thus by \eqref{K_cond} there is a constant $c_1 > 1$ such that $K(x',y') \leq c_1 K(x,y')$.
	
	Moreover:
	\begin{equation*}
		d(x,y) \leq d(x,y') + d(y', y) \leq d(x,y') + 2r(B') \leq (1 + 4/c) d(x,y'),
	\end{equation*}
	
	\noindent thus by \eqref{K_cond} there is a constant $c_2 > 1$ such that $K(x,y') \leq c_2 K(x,y)$. Thus:	
	\begin{equation*}
		K(x',y') \leq c_1 c_2 K(x,y),
	\end{equation*}
	
	\noindent and we have $\varphi(B') \leq c_1 c_2 \varphi(B)$.
\end{proof}

We further assume that $\varphi$ satisfies the following condition:  there is some $\varepsilon > 0$ and some constant $L > 0$ such that for any balls $B_1, B_2$, with $B_1\subset B_2$, we have:

\begin{equation}\label{condition_phi}
	\varphi(B_1)\mu(B_1) \leq L \left(\frac{r(B_1)}{r(B_2)}\right)^\varepsilon \varphi(B_2)\mu(B_2).
\end{equation}

\begin{theorem}[C. Pérez and R.L. Wheeden \cite{PerezWheeden03}]\label{th_perez_wheeden}
	Let $(X,d,\mu)$ be a metric space with a doubling measure $\mu$. Let $T$ be an operator defined by \eqref{T_def} and satisfying $\Kcond{}$, with $\varphi$ satisfying \eqref{condition_phi}. Then there is a constant $C$, depending only on the doubling constant and $p$, such that, for any measurable $f:X \goesto \R$:
	\begin{equation}\label{eq:perez_wheeden}
		\|Tf\|_p \leq C \|M_\varphi f\|_p.
	\end{equation}
\end{theorem}

In addition, for the operator $Tf(x) = \int_M \frac{d(x,y)^s}{\Vol{x}{d(x,y)}} f(y) \diff\mu(y)$, we can replace $M_\varphi$ by the maximal function defined by $M_s f(x) = \sup_{r > 0} r^{s} \fint_{B(x,r)} |f| \diff\mu$. See corollary \ref{cor:riesz_bounded} for the justification.

This theorem is useful, but can't be applied to spaces that are only $R$-doubling. We will now prove a version that we can use in $R$-doubling spaces.

We consider the operator $T_\delta$, $\delta < R$, with kernel $K_\delta(x,y) = K(x,y) \charaset_{\set{d(x,y) < \delta}}$, and we want to compare its $L^p$ norm to that of the maximal function $M_{\varphi,\delta}$ defined by:

\begin{equation}\label{eq:def M_phi_delta}
	M_{\varphi, \delta} f(x) = \sup_{\substack{x \in B\\ r(B) < \delta}} \varphi(B) \int_B |f| \diff\mu.
\end{equation}

The idea of the proof of this comparison will be essentially the same as that of theorem \ref{th_perez_wheeden} given in \cite{PerezWheeden03}, but some care must be taken to account for the different hypotheses properly, and thus we will give the details in what follows.

The hypothesis to prove $\|Tf\|_p \leq C\|M_{\varphi, \delta}f\|_p$ can be weakened compared to those of theorem \ref{th_perez_wheeden}. A key point is that proposition \ref{p_phi_larger_ball} has to hold at least for balls of radius at most $2\delta$. Looking at the proof of the proposition, this is true as long as \eqref{K_cond} holds for $C_2 \leq (1 + 8\rho)$ and $d(x,y) \leq 4\delta$.

Then we take $(X, d, \mu)$ a R-doubling space. $T$ an operator defined by a kernel $K$. We say that $T$, or $K$ verify the condition $\Kcond{\delta}$, if there exist constants $C_1 > 1$, $C_2 \geq 1 + 8 \rho$, such that for any $x,y$ such that $d(x,y) \leq 4\delta$, we have:

\begin{equation}\label{eq:K_cond_cutoff}
\begin{aligned}
	\forall x'\in X,\, d(x', y) \leq C_2 d(x,y),\quad &K(x,y) \leq C_1 K(x', y)\\
	\forall x'\in X,\, d(x, y') \leq C_2 d(x,y),\quad &K(x,y) \leq C_1 K(x, y').
\end{aligned}
\end{equation}

Property $\Kcond{\delta}$ ensure that \ref{p_phi_larger_ball} holds for balls of radius less than $2\delta$.

Since we will end up considering balls of a radius slightly larger than $\delta$, the following proposition will be useful.

\begin{proposition}\label{pr:M_phi_p_bound}
	Let $(X,d,\mu)$ satisfies $\doubling{2(2\kappa+1)\delta}$ for $\delta > 0$, $\kappa > 1$, $T$ an operator satisfying $\Kcond{4(2\kappa +1)\delta}$, and such that the associated functional $\varphi$ satisfies \eqref{condition_phi} when $r(B_1), r(B_2) \leq 2(2\kappa +1) \delta$. Then for any $p \in (1, \infty]$, there is some constant $C$ depending only on $p$, $\kappa$, the doubling constants, and the constants $\alpha$, $L$, $\varepsilon$, in proposition \ref{p_phi_larger_ball} and in \eqref{condition_phi} such that for any non negative $f$, $\|M_{\varphi, \kappa \delta} f\|_p \leq C \|M_{\varphi, \delta} f\|_p$.
\end{proposition}

\begin{proof}
	We have:
	
	\begin{align*}
		M_{\varphi, \kappa \delta} f(x) &= M_{\varphi, \delta} f(x) + \sup_{\substack{x\in B,\\ \delta < r(B) \leq \kappa\delta}} \varphi(B) \int_B |f|\diff\mu \\
		&\leq M_{\varphi,\delta} f(x) + C \sup_{\substack{x\in B,\\ r(B) = \kappa\delta}} \varphi(B) \int_{B(x,2\kappa\delta)} |f|\diff\mu. 
	\end{align*}

Using that for $x \in B$, $B\subset B(x,2r(B)) \subset B(x,2\kappa\delta)$ and that for any ball $B$ with radius greater than $\delta$, by \eqref{condition_phi} (on balls with radius at most $\kappa \delta$), we have:
\begin{equation*}
	\varphi(B) \leq AL \kappa^\eta \varphi\left(\frac{\kappa\delta}{r(B)}B\right).
\end{equation*}

Now, for any ball $B$ containing $x$ with radius equal to $\kappa \delta$. Let $B_x = B(x,\delta)$. For $y \in 2\kappa B_x$, consider the ball $Q(y) = B(y,\delta)$. We have $Q(y) \subset (2\kappa + 1)B_x$, thus using $\doubling{(2\kappa +1)\delta}$, we have that: 
\begin{equation*}
	\mu(2\kappa B_x) \leq A^2(2\kappa + 1)^\eta \mu(Q(y)).
\end{equation*} 

For $y \in (2\kappa +1)B_x$, we also have that $B \subset B(z, 2(2\kappa +1)\delta)$, thus using \eqref{condition_phi} (for balls with radius at most $2(2\kappa +1)\delta)$), $\doubling{2(2\kappa +1)}$ and $\Kcond{4(2\kappa +1)\delta)}$, we get that:
\begin{equation*}
	\varphi(B) \leq A^2 \parenfrac{2(2\kappa +1)}{2\kappa}^\eta \alpha \varphi(Q(y)).
\end{equation*}

Putting all this together, we get:
\begin{align*}
	\varphi(B) \int_{B(x,2\kappa\delta)} |f|\diff\mu &= \varphi(B) \fint_{2\kappa B_x} \mu(2\kappa B_x) |f|\diff\mu \\ 
	&\leq C \varphi(B) \fint_{2\kappa B_x} \mu(Q(y)) |f(y)| \diff\mu(y) \\
	&\leq C \fint_{2\kappa B_x} \varphi(B) \int_{Q(y)} \diff\mu(z) |f(y)|\diff\mu(y) \\
	&\leq C\frac{1}{\Vol{x}{2\kappa\delta}}\int_{(2\kappa + 1)B_x} \varphi(B) \int_{2\kappa B_x \cap B(z,\delta)} |f(y)|\diff\mu(y) \diff\mu(z) \\
	&\leq C A\parenfrac{2\kappa +1}{2\kappa}^\eta \fint_{(2\kappa +1) B_x} \varphi(B(z,\delta)) \int_{B(z,\delta)} |f(y)| \diff\mu(y) \diff\mu(z) \\
	&\leq C \fint_{(2\kappa +1) B_x} M_{\varphi, \delta} f \diff\mu.
\end{align*}

And the constant $C$ depends only on the doubling constants, $L$, $\alpha$ and $\kappa$. Then we have:

\begin{equation}
	M_{\varphi,\kappa\delta} f(x) \leq M_{\varphi,\delta} f(x) + C M_{(2\kappa +1) \delta}\left(M_{\varphi, \delta}f\right)(x).
\end{equation}

The theorem follows from the boundedness of the classical maximal function $M_{(2\kappa +1)\delta}$ on any $L^p$, $p > 1$, under $\doubling{2(2\kappa +1)\delta}$.
\end{proof}

\begin{theorem}\label{cutoff_kernel}
	Let $\delta > 0$. Let $\rho > 0$ be the sidelength constant of dyadic cubes. Suppose that $(X,d, \mu)$ satisfies $\doubling{2(6\rho +1) \delta}$. Assume that $K$ satisfies $\Kcond{4(6\rho +1)\delta}$, and that $\varphi$ satisfies \eqref{condition_phi} for balls with radius at most $2(6\rho+1)\delta$. Let $p \geq 1$. Then there is a constant $C > 0$ (depending only on the doubling constants, $\rho$, $p$ and of the constants in \eqref{condition_phi}, \eqref{K_cond}) such that we have:

	\begin{equation}\label{eq p_bound}
		\int_X |T_\delta f|^p \diff\mu \leq C \int_X (M_{\varphi, \delta} f)^p \diff\mu.
	\end{equation}
\end{theorem}

\begin{proof}
	We will show that there exist some constant $C > 0$ such that for any non negative function $f$, we have $\int_X |T_\delta f|^p \diff\mu \leq C \int_X (M_{\varphi, 3\rho\delta} f)^p \diff\mu$. Then the theorem will follows by proposition \ref{pr:M_phi_p_bound}.
	
	To prove this, we define, for any $m\in \Z$, the operator $T_m$ by:
	
	\begin{equation*}
		T_mf(x) = \int_{d(x,y) > \rho^m} K_\delta (x,y) f(y) \diff\mu(y).
	\end{equation*}
	
	Then, if for any $m \in \Z$, and for any non negative measurable functions $f,g$, we have:
	
	\begin{equation}\label{eq p_dual_bound}
		\int_X T_m f g\diff\mu = \int_{d(x,y) > \rho^m} K_\delta (x,y) f(y) g(x) \diff\mu(x,y) \leq C\|M_{\varphi, 3\delta} f\|_p \|g\|_{p'}.
	\end{equation}
	
	Then by the monotone convergence theorem, taking $m\goesto -\infty$, the same inequality holds but with $T_m$ replaced by $T$, and by duality, \eqref{eq p_bound} is true.
		
	Take $m \in \Z$, and let $f, g$ be non negative measurable functions. Let $\collection{D}_m = \set*{\mathcal{E}_\alpha^k}_{\alpha \in \N^*}^{k \geq m}$ be a decomposition of $X$ in dyadic cubes given by theorem \ref{cubes} with sidelengths $\rho^k$. If $(x,y) \in X$ are such that $d(x,y) > \rho^m$, we take the integer $l \geq m$ such that:
	\begin{equation*}
		\rho^l < d(x,y) \leq \rho^{l+1}.
	\end{equation*}
	
	Let $Q$ be the cube of length $\rho^l$ containing $x$, $B(Q) = B\left(c_Q, \rho^{l+1}\right)$ the containing ball. We recall that $\rho^{-1} B(Q) \subset Q \subset B(Q)$.
	
	We have:
	\begin{equation*}
		d(c_Q, y) \leq d(c_Q, x) + d(x,y) \leq 2\rho^{l+1},
	\end{equation*}
	\noindent thus $y \in 2B(Q)$. Since $d(x,y) > \rho^l = \frac{1}{2\rho} r\left(2B(Q)\right)$, we have by definition of $\varphi$, and by proposition \ref{p_phi_larger_ball}:
	\begin{equation*}
		K(x,y) \leq \varphi(2B(Q)) \leq \alpha\varphi(B(Q))
	\end{equation*}
	
	To apply proposition \ref{p_phi_larger_ball} we need $\Kcond{4\rho\delta}$. 
	
	If we suppose that $\delta \leq \rho^l = \ell(Q)$, then $d(x,y) \geq \delta$ and $K_\delta(x,y) = 0$. 
	
	We have proved that if $Q$ is the cube of length comparable with $d(x,y)$, containing $x$, we have $y \in 2B(Q)$ and:	
	\begin{equation*}
		K_\delta(x,y) \leq C\varphi(B(Q))\charaset_{\set*{R \in \collection{D}_m,\, \ell(R) < \delta}}(Q)\charaset_{Q}(x) \charaset_{2B(Q)}(y).
	\end{equation*}
	
	If $r$ is the largest integer such that $\rho^r < \delta$, define $\collection{D}_m^r = \set*{\mathcal{E}_\alpha^k;\; m \leq k \leq r}$. For any $x, y \in X$ with $d(x,y) > \rho^m$, there is at least one cube $Q\in \collection{D}_m$ such that the previous inequation holds, and since both sides of it are zero if $\ell(Q) \geq \delta$, we have, for any $x,y \in X$:
	
	\begin{equation*}
		K_\delta(x,y) \leq \sum_{Q \in \collection{D}_m^r} C \varphi(B(Q))\charaset_{Q}(x) \charaset_{2B(Q)}(y).
	\end{equation*}
	
	And so, for any $f, g \geq 0$:
	
	\begin{equation*}
		\int_X T_mf g \diff\mu\leq C\sum_{Q \in \collection{D}_m^r} \varphi(B(Q)) \int_{2B(Q)} f \diff\mu \int_Q g \diff\mu.
	\end{equation*}
	
	But for any fixed integer $k \geq m$, the cubes of length of length $\rho^k$, $\set*{\mathcal{E}_\alpha^k}$ are pairwise disjoints, and $X = \bigcup_\alpha \mathcal{E}_\alpha^k$. Then using this decomposition for $k = r$,
	
	\begin{equation*}
		\int_X T_mf g \diff\mu\leq 
		C \sum_{\alpha \geq 1} \sum_{\substack{Q \in \collection{D}_m^r\\ Q \subset \mathcal{E}_\alpha^r}} \varphi(B(Q)) \int_{2B(Q)} f \diff\mu \int_Q g \diff\mu.
	\end{equation*}
	
	Then for a constant $\gamma \geq 1$ to be determined, for any $\alpha \geq 1$, and $n \in \Z$, define:
	
	\begin{equation}
		\collection{C}_\alpha^n = \set*{Q \in \collection{D}_m^r, Q \subset \mathcal{E}_\alpha^r;\; \gamma^n < \frac{1}{\mu(B(Q))}\int_Q g \diff \mu \leq \gamma^{n+1}}.
	\end{equation}
	
	We let $n_\alpha$ be the unique integer such that $\mathcal{E}_\alpha^r \in \collection{C}_\alpha^{n_\alpha}$. Notice that $\set*{\collection{C}_\alpha^{n}}_{n\in \Z}$ is a partition of $\set{Q \in \collection{D}_m^r;\, Q \subset \mathcal{E}_\alpha^r}$. Then we have:
	
	\begin{equation*}
		\int_X T_mf g \diff\mu \leq C \sum_{\alpha \geq 1} \sum_{n\in \Z} \gamma^{n+1} \sum_{Q \in \collection{C}_\alpha^n} \varphi(B(Q)) \mu(B(Q)) \int_{2B(Q)} f \diff\mu .
	\end{equation*}
	
	For any $\alpha \geq 1$, we let $\set*{Q_{j,\alpha}^n}_{j\in J_n}$, for some index set $J_n$, be the collection of the maximal dyadic cubes subset of $\mathcal{E}_{\alpha}^r$ such that:
	\begin{equation*}
		\gamma^n < \frac{1}{\mu\left(B\left(Q_{j,\alpha}^n\right)\right)} \int_{Q_{j,\alpha}^n} g \diff \mu.
	\end{equation*}
	
	If $n \leq n_\alpha$, then there is exactly one such maximal cube: $\mathcal{E}_\alpha^r$. Also, the function $(n,Q) \mapsto Q$ is an injection from the set of the couples $(n, Q)$ with $n \leq n_\alpha$, $Q \in \collection{C}_\alpha^n$ to $\set*{Q\in D_m^r: Q \subset \mathcal{E}_\alpha^r}$, thus:	
	\begin{small}
	\begin{equation*}
		\sum_{n \leq n_\alpha} \sum_{Q\in \collection{C}_\alpha^n} \gamma^{n+1}  \varphi(B(Q))\mu(B(Q)) \int_{2B(Q)} f \diff\mu \\
		\leq \gamma^{n_\alpha + 1} \sum_{\substack{Q\in \collection{D}_m^r\\ Q \subset \mathcal{E}_\alpha^r}} \varphi(B(Q))\mu(B(Q)) \int_{2B(Q)} f \diff \mu.
	\end{equation*}
	\end{small}
	
	If $n > n_\alpha$, then any $Q_{j,\alpha}^n$ is a strict subset of $\mathcal{E}_\alpha^r$. For such a maximal cube $\mathcal{F}$, we let $P$ be his dyadic parent i.e.\ the only cube of length $\rho \ell(\mathcal{F})$ containing $P$. We have $P \subset \mathcal{E}_\alpha^r$, and by using the maximality of $\mathcal{F}$, and that $B(\mathcal{F}) \subset 2B(P)$, and using the $\rho \delta$-doubling ($B(P)$ has radius less than $\rho \delta$):
	\begin{equation}\label{eq_max_cubes_gamma}
		\gamma^n < \frac{1}{\mu(B(\mathcal{F}))} \int_\mathcal{F} g \diff \mu \leq \frac{\mu(B(P))}{\mu(B(\mathcal{F}))} \frac{1}{\mu(B(P))}\int_P g \diff \mu \leq C \rho^\eta \gamma^n = \kappa \gamma^n,
	\end{equation}
	
	\noindent with the constant $\kappa$ depending only on $\rho$ and on the doubling constant. Then choosing $\gamma > \kappa$, we have:
	\begin{equation*}
		\frac{1}{\mu(B(\mathcal{F}))} \int_\mathcal{F} g \diff \mu \leq \gamma^{n+1},
	\end{equation*}
	
	\noindent thus $\mathcal{F} \in \collection{C}_\alpha^n$. Thus for a fixed $n > n_\alpha$, every cube in $\collection{C}_\alpha^n$ is in a (unique) $Q_{j,\alpha}^n$, which are disjoint in $j$ by maximality. Thus, writing $Q_{j,\alpha}^{n_\alpha}$ for $\mathcal{E}_\alpha^r$ we have:
	
	\begin{equation*}
		\int_X \left(T_mf\right) g \diff\mu \leq C \sum_{\alpha \geq 1} \sum_{n \geq n_\alpha} \gamma^{n+1} \sum_{j \in J_n}  \sum_{\substack{Q \in \collection{D}_\alpha^m\\Q\subset Q_{j,\alpha}^n}} \varphi(B(Q)) \mu(B(Q)) \int_{2B(Q)} f \diff\mu.
	\end{equation*}
	
	Now we use the following lemma (see lemma $6.1$ of \cite{PerezWheeden03}):
	
	\begin{lemma}
		Let $(X,d,\mu)$ satisfies $\doubling{\delta}$. Let $\varphi$ be a functional on balls that satisfies \eqref{condition_phi} for balls of radius at most $\rho \delta$. Then there is a constant $C$ depending only on the constant $L$ of \eqref{condition_phi} and on the  doubling constant such that for any $f \geq 0$ and any dyadic cube $Q_0 \in \collection{D}_m^r$, with $\rho^r \leq \delta$, 
		
		\begin{equation}
			\sum_{\substack{Q\in \collection{D}_m \\ Q \subset Q_0}} \varphi(B(Q))\mu(B(Q)) \int_{2B(Q))} f \diff \mu \leq C \varphi(B(Q_0))\mu(B(Q_0)) \int_{3B(Q_0)} f \diff \mu.
		\end{equation}
	\end{lemma}
	
	\begin{proof}
		By \eqref{condition_phi}, we have:
		
		\begin{small}
		\begin{align}
			\sum_{\substack{Q\in \collection{D}_m \\ Q \subset Q_0}} \varphi(B(Q))\mu(B(Q)) \int_{2B(Q))} f \diff \mu 
			&\leq L \varphi(B(Q_0))\mu(B(Q_0)) \sum_{\substack{Q\in \collection{D}_m \\ Q \subset Q_0}} \left(\frac{\ell(Q)}{\ell(Q_0)}\right)^\varepsilon \int_{2B(Q))} f \diff \mu \nonumber \\
			&\leq L \varphi(B(Q_0))\mu(B(Q_0)) \sum_{l= 0}^{+\infty} \rho^{-\varepsilon l}  \sum_{\substack{Q\in \collection{D}_m \\ Q \subset Q_0 \\ \ell(Q) = \rho^{-l}\ell(Q_0)}} \int_{2B(Q))} f \diff \label{eq:lemma eq 1}\mu.
		\end{align}
		\end{small}
		
		Then for $Q \in \collection{D}_m, Q \subset Q_0$, and $\ell(Q) \leq \ell(Q_0)$ we have $2B(Q) \subset 3B(Q_0)$. Indeed, if $y \in 2B(Q)$, then:
		
		\begin{align*}
			d(y,x_{Q_0})
			&\leq d(y, x_{Q}) + d(x_Q, x_{Q_0}) \\
			&\leq 2 r(B(Q)) + r(B(Q_0)) \\
			&\leq 3 r(B(Q_0)).
		\end{align*}
		
		Thus, the left hand side of \eqref{eq:lemma eq 1} is less than:
		
		\begin{equation*}
		L \varphi(B(Q_0))\mu(B(Q_0)) \int_{3B(Q_0))}  f(x)  \sum_{l=0}^\infty \rho^{-\varepsilon l} \sum_{\substack{Q\in \collection{D}_m \\ Q \subset Q_0 \\ \ell(Q) = \rho^{-l}\ell(Q_0)}}
		 \charaset_{2B(Q)}(x) \diff \mu(x)		.
		\end{equation*}				
		
		Then it suffices to show that for each $l$, any $x$ of $3B(Q_0)$ is in at most $N$ of the $2B(Q)$, with $\ell(Q) = \rho^{-l}\ell(Q_0)$, with $N$ independant of the choices of $x$ and $Q_0$. For $l = 0$, there is only one $Q$: $Q_0$ itself, and thus it is true.
		
		Now fix $l > 1$, let $x\in M$, and $Q$ be a cube of sidelength $\rho^{-l}\ell(Q_0)$ such that $x\in 2B(Q)$. We write $\ell = \ell(Q) \leq \rho^{-1}\delta$. Then for $y \in Q$: 
		\begin{equation*}
			d(x, y) \leq d(x, x_Q) + d(y, x_Q) \leq 3\rho \ell\leq 3\delta,
		\end{equation*}
		
		\noindent then we have $B(x_Q, \ell) \subset Q \subset B(x, 3\rho\ell)$. By the proposition \ref{p_covering_card}, then there can be at most $N$ disjoint balls of radius $\ell \leq \delta$ with center in a ball of radius $3\rho\ell$, with the constant $N$ depending only on $\rho$ and on the $\delta$-doubling constant.
				
		Thus:		
		\begin{equation*}
			\sum_{l=0}^\infty \rho^{-\varepsilon l} \sum_{\substack{Q\in \collection{D}_m \\ Q \subset Q_0 \\ \ell(Q) = \rho^{-l}\ell(Q_0)}} 1 \leq N \frac{1}{1- \rho^{-\varepsilon}},
		\end{equation*}
		
		\noindent and the lemma follows.		
	\end{proof}
	
	Then applying the lemma:
	\begin{equation*}
		\int_X \left(T_mf\right) g \diff\mu \leq C \sum_{\alpha \geq 1} \sum_{n \geq n_\alpha} \gamma^{n+1} \sum_{j \in J_n} 
		\varphi\left(B\left(Q_{j,\alpha}^n\right)\right)
		\mu\left(B\left( Q_{j,\alpha}^n \right) \right) 
		\int_{3B\left( Q_{j,\alpha}^n \right)} f \diff \mu.
	\end{equation*}
	
	And thus since $Q_{j,\alpha}^n \in \collection{C}_\alpha^n$, $\gamma^n \leq \frac{1}{\mu\left(B(Q_{j,\alpha}^n)\right)} \int_{Q_{j,n}^\alpha} g \diff\mu$, and so,
	
	\begin{equation*}
		\int_X \left(T_mf\right) g \diff\mu \leq C \gamma \sum_{\alpha \geq 1} \sum_{n \geq n_\alpha} \sum_{j \in J_n} \varphi\left(B(\left(Q_{j,\alpha}^n\right)\right) 
		\int_{3B\left( Q_{j,\alpha}^n \right)} f \diff \mu \int_{Q_{j,\alpha}^n} g \diff \mu,
	\end{equation*}
	
	and we have:
	
	\begin{equation}
		\int_X \left(T_mf\right) g \diff \mu \leq c \sum_{\alpha, n, j} \varphi\left(B\left(Q_{j,\alpha}^n\right)\right)
		\mu\left(Q_{j,\alpha}^n\right) \int_{3B(Q_{j,\alpha}^n)} f \diff \mu \frac{1}{\mu(Q_{j\alpha}^n)} \int_{Q_{j,\alpha}^n} g \diff \mu.
	\end{equation}
	
	Then using Hölder's inequality, and that by \eqref{condition_phi} there is some constant c depending only on $\alpha, A, L, \varepsilon$ such that $ \varphi(B) \leq c\varphi(3B)$ (ball of radius $3\rho \delta$), we get:
	
	\begin{samepage}
	\begin{multline*}
		\int_X \left(T_mf\right) g \diff \mu \leq  
		C\left(\sum_{\alpha, n, j}\mu\left(Q_{j,\alpha}^n \right)
		\left( \varphi\left(B\left(3Q_{j,\alpha}^n\right)\right) \int_{3B(Q_{j,\alpha}^n)}  f \diff \mu
		\right)^p		
		\right)^\frac{1}{p}  \\
		\left(\sum_{\alpha, n, j}\mu\left(Q_{j,\alpha}^n\right) 
		\left( \frac{1}{\mu(Q_{j\alpha}^n)} \int_{Q_{j,\alpha}^n} g \diff \mu
		\right)^{p'} 
		\right)^\frac{1}{p'}.
	\end{multline*}
	\end{samepage}
	
	Now we just need to establish a majoration of $\mu(Q_{j,\alpha}^n)$ by a constant time the measure of a set $E_{j,\alpha}^n$, with the $E_{j,\alpha}^n$ being pairwise disjoint in $j, n, \alpha$. For this, define $\Omega_\alpha^n$ by 
	
	\begin{equation}
		\Omega_\alpha^n = \set*{ x \in \mathcal{E}_\alpha^r;\; \sup_{\substack{ Q \in \collection{D}_m^r\\ x \in Q}} \frac{1}{\mu(B(Q))} \int_Q g \diff \mu > \gamma^n} = \bigcup_{j\in J_n} Q_{j,\alpha}^n,
	\end{equation}
	
	and define the set $E_{j,\alpha}^n = Q_{j,\alpha}^n \setminus \Omega_\alpha^{n+1}$. We have that $E_{j,\alpha}^n \subset \Omega_\alpha^n \setminus \Omega_\alpha^{n+1}$, and the $E_{j,\alpha}^n$ are pairwise disjoints in $j, n, \alpha$.
	
	Now we want to show that for $\gamma$ chosen large enough, $\mu(Q_{j,\alpha}^n) \leq 2\mu(E_{j,\alpha}^n)$.
	
	First: 
	\begin{equation*}
		Q_{j,\alpha}^n \cap \Omega_{\alpha}^{n+1} = \bigcup_i  \left(Q_{j,\alpha}^n \cap Q_{i,\alpha}^{n+1}\right),
	\end{equation*}
	\noindent but we have:
	\begin{equation*}
		\frac{1}{\mu\left(B\left( Q_{i,\alpha}^{n+1} \right)\right)} \int_{Q_{i,\alpha}^{n+1}} g \diff \mu > \gamma^{n+1} > \gamma^n,
	\end{equation*} 
	\noindent thus by maximality of $Q_{j,\alpha}^{n}$, and by the properties of dyadic cubes, etiher $Q_{i,\alpha}^{n+1} \subset Q_{j,\alpha}^n$ or $Q_{j,\alpha}^n \cap Q_{i,\alpha}^{n+1} = \emptyset$. Hence:	
	\begin{equation*}
		\mu\left( Q_{j,\alpha}^n \cap \Omega_\alpha^{n+1} \right) = \sum_{i: Q_{j,\alpha}^n \cap Q_{i,\alpha}^{n+1} = \emptyset} \mu \left( Q_{j,\alpha}^n \cap Q_{i,\alpha}^{n+1}\right) = \sum_{i: Q_{i,\alpha}^{n+1} \subset Q_{j,\alpha}^n} \mu\left(Q_{i,\alpha}^{n+1}\right),
	\end{equation*}
	
	\noindent but:	
	\begin{equation*}
		\mu\left(Q_{i,\alpha}^{n+1}\right) \leq \mu\left(B\left(Q_{i,\alpha}^{n+1}\right)\right) \leq \gamma^{-n - 1} \int_{Q_{i,\alpha}^{n+1}} g \diff \mu,
	\end{equation*}
	
	\noindent and since the $Q_{i,\alpha}^{n+1}$ considered are disjoints and subsets of $Q_{j,\alpha}^n$, we have:
	\begin{equation*}
		\mu(Q_{j,\alpha}^n \cap \Omega_\alpha^{n+1}) \leq \gamma^{-n-1} \int_{Q_{j,\alpha}^n} g \diff \mu \leq \kappa \gamma^{-1} \mu(B(Q_{j,\alpha}^n)),
	\end{equation*}
	
	\noindent where $\kappa$ is the constant in \eqref{eq_max_cubes_gamma}. But we have:	
	\begin{equation*}
		\mu(Q_{j,\alpha}^n) = \mu(E_{j,\alpha}^n) + \mu(Q_{j,\alpha}^n\cap \Omega_{\alpha}^{n+1}),
	\end{equation*}
	
	\noindent and so choosing $\gamma = 2\kappa$, it follows that:	
	\begin{equation*}
		\mu\left(Q_{j,\alpha}^n\right) \leq \frac{\gamma}{\gamma - \kappa} \mu\left(E_{j,\alpha}^n \right) = 2\mu\left(E_{j,\alpha}^n \right).
	\end{equation*}
	
	Consequently, we have:	
	\begin{samepage}
	\begin{multline*}
		\int_X \left(T_mf\right) g \diff \mu \leq 
		2C\left(\sum_{\alpha, n, j}\mu\left(E_{j,\alpha}^n \right)
		\left( \varphi\left(B\left(3Q_{j,\alpha}^n\right)\right) \int_{3B(Q_{j,\alpha}^n)}  f \diff \mu
		\right)^p		
		\right)^\frac{1}{p}\\
		\left(\sum_{\alpha, n, j}\mu\left(E_{j,\alpha}^n\right) 
		\left( \frac{1}{\mu(Q_{j\alpha}^n)} \int_{Q_{j,\alpha}^n} g \diff \mu
		\right)^{p'} 
		\right)^\frac{1}{p'},
	\end{multline*}
\end{samepage}	
	
	\noindent but since $E_{j,\alpha}^n \subset Q_{j,\alpha}^n$, it follows that:	
	\begin{equation*}
		\mu\left(E_{j,\alpha}^n \right) \left( \varphi\left(B\left(3Q_{j,\alpha}^n\right)\right) \fint_{3B\left(Q_{j,\alpha}^n\right)} f \diff \mu \right)^p
	\leq \int_{E_{j,\alpha}^n} \left(M_{\varphi, 3\rho^{r+1}} f\right)^p \diff\mu,
	\end{equation*} 
	
	\noindent and a similar inequality for the integral on $g$. In addition using that the $E_{j,\alpha}^n$ are pairwise disjoint, and that $\rho^r < \delta$, we get:
	\begin{equation}
		\int_{X} \left(T_mf\right) g \diff \mu \leq 2C\left( \int_{X} \left(M_{\varphi, 3\rho\delta } f\right)^p \diff\mu\right)^\frac{1}{p} \left( \int_X \left( M_{d, \delta} g \right)^{p'} \diff\mu \right)^\frac{1}{p'}.
	\end{equation}
	
	Now, using proposition \ref{p_dyad_max}, for all $f, g \geq 0$, there is a constant $C$ depending only on $p, A, \alpha, \varepsilon$ (specifically it depends on the constants for the $\rho \delta$-doubling) such that:
	\begin{equation*}
		\int_{X} \left(T_mf\right) g \diff \mu \leq C \left\|M_{\varphi, 3\rho\delta} f \right\|_p \|g\|_{p'}.
	\end{equation*}
	
	This holds under $\doubling{r\delta}$, $\Kcond{2\rho\delta}$ and the fact that \eqref{condition_phi} holds for balls of radius at most $3\rho\delta$. The stronger hypotheses are what we need to apply proposition \ref{pr:M_phi_p_bound} which gives us:
	\begin{equation}
		\int_{X} \left(T_mf\right) g \diff \mu \leq C \left\|M_{\varphi, \delta} f \right\|_p \|g\|_{p'},
	\end{equation}
	
	\noindent which proves the theorem.
\end{proof}

	Finally we have the theorem applied to the operators which will be of interest to us:
	\begin{corollary}\label{cor:riesz_bounded}
		Let $\mu$ be a measure satisfying $\doubling{R}$ and $\rd{R}$, with $R > 0$, $\eta \geq \nu > 0$ ($\eta \geq \nu$ is automatic). Let $s \leq \nu$. Let $\delta \leq R$. If $K(x,y) = \frac{d(x,y)^s}{\Vol{x}{d(x,y)}}$, then the associated operator $T_\delta$ satisfies the hypotheses of theorem \ref{cutoff_kernel}. Moreover, the theorem still holds with $M_{\varphi, \delta}f$ replaced by the following maximal function:
		\begin{equation}
			M_{s, \delta} f(x) = \sup_{0 < r < \delta} r^{s}\fint_{B(x,r)} |f|\diff\mu .
		\end{equation}
	\end{corollary}
	
	\begin{proof}
		First, take some $b > 1$, by proposition \ref{pr:reverse_doubling bigger R}, $\mu$ is $bR$-reverse doubling of order $\nu$. Then, we must verify that $K$ satisfies the hypotheses of theorem \ref{cutoff_kernel}. Let $d(x,y) \leq R$ and $d(x, y') \leq b d(x,y)$, then we have by doubling and reverse doubling,:
		
		\begin{align*}
			\frac{1}{\Vol{x}{d(x,y)}} &\leq \frac{1}{\Vol{x}{d(x,y')}} \frac{\Vol{x}{b d(x,y)}}{\Vol{x}{d(x,y)}} \frac{\Vol{x}{d(x,y')}}{\Vol{x}{b d(x,y)}}, \\
			&\leq C b^{\eta-\nu} \left( \frac{d(x,y')}{d(x,y)}\right)^\nu \frac{1}{\Vol{x}{d(x,y'}}.
		\end{align*}
		
		Thus, provided that $s \leq \nu$:
		
		\begin{equation*}
			K(x,y) \leq C b^{\eta - \nu} \left(\frac{d(x,y')}{d(x,y)}\right)^{\nu -s} K(x,y') \leq C b^{\eta - s} K(x,y').
		\end{equation*}
		
		Furthermore, if $d(x',y) \leq \alpha d(x,y)$, using the doubling property, there are $c, C$ such that $c \Vol{y}{d(x',y)} \leq \Vol{x'}{d(x',y)} \leq C \Vol{y}{d(x',y)}$, and so doing the same calcuations we have:
		
		\begin{equation*}
			K(x,y) \leq C b^{\nu - s} K(x',y).
		\end{equation*}
		
		And there are $C_1, C_2 > 1$ such that \eqref{K_cond} is satisfied.
			
		Then, using the definition of $\varphi$ and doubling:
		\begin{equation*}
			c \frac{r(B)^s}{\mu(B)} \leq \varphi(B) \leq C \frac{r(B)^s}{\mu(B)},
		\end{equation*}
		
		\noindent for some constants that depends only on $s, \rho$ and the doubling constant. Then since we have, for $B_1 \subset B_2$: 
		\begin{equation*}
			r(B_1)^s \leq 2^s r(B_2)^s,
		\end{equation*}
		\noindent we easily verify that $\varphi$ satisfies \eqref{condition_phi} with $\varepsilon = s$.
				
		Then it is enough to prove that the centered and uncentered version of the maximal function $M_{s,\delta}$ are equivalent in $L^p$ norms. This follow from the same argument as that of proposition \ref{pr:uncentered equiv centered}. 
	\end{proof}

\section{Relative Faber-Krahn inequality and estimates on the heat kernel and the Riesz and Bessels potentials}\label{sec:Faber Krahn}

\subsection{Faber-Krahn and doubling}

	The results from this subsection are due to A.A. Grigor'yan \cite{Grigoryan94, Grigor'yanBook09}, or are slight adaptation of his results to the R-doubling case.
	
\begin{theorem}\cite{Grigor'yanBook09}
	Let $(M,g,\mu)$ be a weighted manifold, and let $\set*{B(x_i, r_i)}_{i\in I}$ be a family of relatively comapct balls in M, where $I$ is an arbitrary index set. Assume that, for any $i\in I$, $U \subset B(x_i,r_i)$, there is a constant $a_i > 0$ such that the following Faber-Krahn inequality holds:
	\begin{equation}
		\lambda_1(U) \geq a_i \mu(U)^{-2/\eta}.
	\end{equation}
	
	Let $\Omega = \bigcup_{i\in I} B\left(x_i, \frac{r_i}{2}\right)$. Then for all $x,y \in \Omega$ and $t \geq t_0 > 0$, we have:
	\begin{equation}
		p_t(x,y) \leq \frac{C(\eta) \left(1 + \frac{d(x,y)^2}{t}\right)^{\eta/2} \exp\left(-\frac{d(x,y)^2}{4t} - \lambda_1(M) (t - t_0)\right)}{\left(a_i a_j \min\left(t_0, r_i^2\right) \min\left(t_0, r_j^2\right)\right)^{\eta/4}},
	\end{equation}
	
	\noindent where $i, j$ are the indices such that $x\in B\left(x_i, \frac{r_i}{2}\right)$ and $y \in B\left(x_j, \frac{r_j}{2}\right)$.
\end{theorem}

On a manifold which admits $\FK{R}$, applying this theorem with the family of all balls of radius less than $R$, $\set*{B(x,r)}_{\substack{x\in M,\\ 0 < r \leq R}}$, with $a_{x,r} = \frac{b}{r^2} \Vol{x}{r}^{2/\eta}$, $t_0 = t$, and $r = \sqrt{t}$, when $t \leq R^2$ we get:

\begin{align*}
	p_t(x,y) &\leq C(\eta) \frac{  \left(1 + \frac{d(x,y)^2}{t}\right)^{\eta/2} e^{-\frac{d(x,y)^2}{4t}}}{\left(a_{x,\sqrt{t}}b_{y,\sqrt{t}} t^2 \right)^{\eta/4}}, \\
	&\leq \frac{C(\eta)}{b^{\eta/2}} \frac{ e^{-\frac{d(x,y)^2}{ct}}}{\Vol{x}{\sqrt{t}}^{1/2} \Vol{y}{\sqrt{t}}^{1/2}}.
\end{align*}

If $t > R^2$, then we do the same thing, but with $r = R$, and we obtain the following:
\begin{theorem}\label{th:heat_kernel_estimate}
	Let $(M,g,\mu)$ be a weighted Riemannian manifold, suppose that there is $R > 0$ such that $M$ satisfies $\FK{R}$. Then $\mu$ satisfies $\doubling{R
}$, and for any $c > 4$ there is some constant $K > 0$ such that the heat kernel has the following upper bounds:
	\begin{align}
		p_t(x,y) &\leq \frac{K}{\Vol{x}{\sqrt{t}}^{1/2}\Vol{y}{\sqrt{t}}^{1/2}} e^{-\frac{d(x,y)^2}{ct}},\quad t \leq R^2 \\
		p_t(x,y) &\leq \frac{K}{\Vol{x}{R}^{1/2}\Vol{y}{R}^{1/2}} e^{-\frac{d(x,y)^2}{ct}},\quad t > R^2.
	\end{align}
	
	The constant $K$ depends only on $b$ and $\eta$ in the Faber-Krahn inequality and on the $c > 4$ chosen.
\end{theorem}

The estimate on the heat kernel follows from Theorem 5.2 of \cite{Grigoryan94}. The R-doubling follow from the proof of Proposition 5.2 of the same article.	
	
Conversely, we have:

\begin{proposition}\cite{Grigoryan94} \label{pr:heat + doubling implies FKR}
	Let $(M, g, \mu)$ be a complete, weighted Riemannian manifold. If $\mu$ satisfies $\doubling{R}$, if for any $x\in M$, the annuli $B(x,r') \setminus B(x,r)$, for $0 \leq r < r' \leq R$ are non-empty, and if there is some constant $B$ such the heat kernel satisfies:
	\begin{equation}\label{eq:diagonal estimate}
		p_t(x,x) \leq \frac{B}{\Vol{x}{\sqrt{t}}},
	\end{equation}
	\noindent for all $x \in M$, and for all $0 < t \leq R^2$, then there is some constant $\kappa \in (0, 1)$, depending only on the doubling and reverse doubling constants, such that $M$ admits a relative Faber-Krahn inequality at scale $\kappa R$, with $\eta$ being the doubling order and $b$ depending only on $A, B$, and $\kappa$ depends only on the doubling constants and on $B$.
\end{proposition}

\begin{proof}
	This is a modification of the proof in \cite{Grigoryan94}, to take into account the $R$ doubling case.
	
	Fix a ball $B(x,r)$, with $r < R$, and let $U$ be an open relatively compact subset of $B(x,r)$. Using the doubling volume property, we have, if $t \leq r^2$:
	\begin{equation*}
		e^{-\lambda_1(U)t} \leq \int_U p_t(y,y) \diff\mu(y) \leq B\int_U \frac{\diff\mu(y)}{\Vol{y}{\sqrt{t}}} \leq AB \frac{\mu(U)}{\Vol{x}{r}} \left( \frac{r}{\sqrt{t}}\right)^\eta,
	\end{equation*}
	
	\noindent thus we have:		
	\begin{equation*}
		\lambda_1(U) \geq \frac{1}{t} \log\left(\frac{1}{AB} \frac{\Vol{x}{r}}{\mu(U)} \left(\frac{\sqrt{t}}{r}\right)^\eta \right).
	\end{equation*}
	
	Choose $t$ such that the logarithm in the above inequality is equal to $1$, i.e.\
	\begin{equation*}
		t = r^2 \left(e AB \frac{\mu(U)}{\Vol{x}{r}} \right)^{2/\eta},
	\end{equation*}
	
	\noindent the condition $t \leq r^2$ then impose $\mu(U) \leq \frac{1}{e AB} \Vol{x}{r}$. For such $U$, we then have:
	\begin{equation}\label{eq:lambda_1_mu}
		\lambda_1(U) \geq \frac{\left(e AB\right)^{-2/\eta}}{r^2} \left( \frac{\Vol{x}{r}}{\mu(U)} \right)^{2/\eta}.
	\end{equation}
	
	Now since the measure $\mu$ satisfies $\doubling{R}$, and since the annuli of radius less than $R$ are non empty, it satisfies $\rd{R}$ for some $\nu > 0$. There is some constant $a \in (0,1)$ such that for any $0 < r < r' \leq R$ we have:
	\begin{equation*}
		\Vol{x}{r} \leq a \left(\frac{r}{r'}\right)^\nu \Vol{x}{r'},
	\end{equation*}
	\noindent with $\nu = -\log_2 a$.
	
	Then for $\kappa = (aeAB)^{-1/\nu}$, if $r \leq \kappa R$, choose $r' = \kappa^{-1} r$. We have for all $U$ relatively compact open subset of $B(x,r)$:
	\begin{equation*}
		\mu(U) \leq \frac{1}{e AB} \mu(B(x,r')),
	\end{equation*}
	
	\noindent thus we can apply \eqref{eq:lambda_1_mu}. Then using by R-reverse doubling $\Vol{x}{r'} \geq a^{-1} \kappa^{-\nu} \Vol{x}{r}$, we have:
	
	\begin{equation*}
		\lambda_1(U) \geq \frac{b}{r^2}\left(\frac{\Vol{x}{r}}{\mu(U)}\right)^{2/\eta}.
	\end{equation*}
	
	With $b = \kappa^2$.
\end{proof}

\subsection{An estimate on the heat kernel}

\begin{proposition}\label{pr:estimate on the heat kernel}
	Let $(M,g,\mu)$ be a complete weighted manifold satisfying $\FK{R}$ for $R > 0$, then for any $c > 4$, there are constants $K_1, K_2, K_3 > 0$ and $\alpha > 0$ such that the following estimates on the heat kernel hold:
	
	If $0 \leq t \leq R^2$, then:
	\begin{equation}
		p_t(x,y) \leq \frac{K_1}{\Vol{x}{\sqrt{t}}}e^{-\frac{d(x,y)^2}{ct}}.
	\end{equation}
	
	If $t > R^2$ and $d(x,y) \leq R$, then:
	\begin{equation}
		p_t(x,y) \leq \frac{K_2}{\Vol{x}{R}}e^{-\frac{d(x,y)^2}{ct}}.
	\end{equation}
	
	If $t > R^2$ and $d(x,y) > R$, then:
	\begin{equation}
		p_t(x,y) \leq \frac{K_3}{\Vol{x}{R}}e^{\alpha \frac{t}{R^2}}e^{-\frac{d(x,y)^2}{ct}}.
	\end{equation}
\end{proposition}

\begin{proof}
	Using the $R$-doubling, we have that for any $t > 0$,

	\begin{equation}
		\Vol{x}{\sqrt{t}} \leq C \Vol{y}{\sqrt{t}} e^{D \frac{d(x,y)}{\sqrt{t}}},
	\end{equation}
	
	\noindent and so, for $t \leq R^2$, we get, for any $c' > 4$:
	\begin{equation}
		p_t(x,y) \leq \frac{CK}{\Vol{x}{\sqrt{t}}} e^{\frac{D}{2}\frac{d(x,y)}{\sqrt{t}} - \frac{d(x,y)^2}{c't}},
	\end{equation}
	
	\noindent so taking $c' < c$, there's some constant $K_1$ such that:
	\begin{equation}
		p_t(x,y) \leq \frac{K_1}{\Vol{x}{\sqrt{t}}} e^{-\frac{d(x,y)^2}{ct}}.
	\end{equation}
	
	When $t > R^2$ and $d(x,y) \leq R$, then the $R$-doubling property for small balls immediately lead to the desired result.
	
	When $d(x,y) > R$, then by the $R$-doubling we obtain, for any $c' > 4$:
	
	\begin{equation*}
	p_t(x,y) \leq \frac{CK}{\Vol{x}{R}}e^{\frac{D}{2}\frac{d(x,y)}{R} - \frac{d(x,y)^2}{c't}}.
	\end{equation*}

	We have that $\frac{D d(x,y)}{2R} - \frac{d(x,y)^2}{c't} \leq \frac{c' c D^2 t}{16 R^2 (c - c')} - \frac{d(x,y)^2}{ct}$, thus there is some constants $K_3$, $\alpha$ which depend on the doubling constant and the choice of $c, c'$, such that:
	
	\begin{equation}
		p_t(x,y) \leq \frac{K_3}{\Vol{x}{R}} e^{\alpha \frac{t}{R^2} - \frac{d(x,y)^2}{ct}}
	\end{equation}
\end{proof}

\subsection{Estimation of the Riesz potential}
	Let $s > 0$, and define the Riesz potential to be the operator $I_s = \Delta^{-s/2}$ on $L^{2}(M,\mu)$. Define $i_s(x,y)$ by:
	
	\begin{equation}\label{eq:Riesz Kernel}
		i_s(x,y) = \frac{1}{\Gamma\left(\frac{s}{2}\right)}\int_0^{+\infty} t^{s/2 - 1} p_t(x,y)dt.
	\end{equation}
	
	Whenever $i_s$ is finite for all $x,y\in M$, it is the Schwartz Kernel of the Riesz potential: in such case, for any $f\in \SmoothComp{M}$, $f$ is in the domain of $I_s$ and:
	\begin{equation}
		I_sf(x) = \int_M i_s(x,y) f(y) d\mu(y),
	\end{equation}
	
	\noindent we thus call $i_s$ the Riesz kernel. A sufficient condition for the Riesz Kernel to be defined is given in the following proposition, which also yields an estimate on it:
	
	\begin{proposition}\label{pr:riesz_potential}
		Let $(M,g,\mu)$ be a manifold satisfying $\FK{}$ and $\rd{}$, $\nu > 0$. Then for any $s < \nu$, there is a constant $C$ depending only on the Faber-Krahn and reverse doubling constants, such that the following inequality holds:
		\begin{equation}
			i_s(x,y) \leq C\frac{d(x,y)^s}{\Vol{x}{d(x,y)}}
		\end{equation}
	\end{proposition}
	
\begin{proof}
	Using proposition \ref{pr:estimate on the heat kernel} when the manifold satisfies $\FK{}$, there is $C > 0$ such that for all $x,y \in M$, $t > 0$:
	\begin{equation}
		p_t(x,y) \leq \frac{K_1}{\Vol{x}{\sqrt{t}}}\exp\left(-\frac{d(x,y)^2}{5t}\right),
	\end{equation}
	
	\noindent and so:
	\begin{equation}
		i_s(x,y) \leq C_s \int_0^\infty \frac{t^{s/2-1}}{\Vol{x}{\sqrt{t}}}e^{-\frac{d(x,y)^2}{5t}} dt.
	\end{equation}
	
	Using the doubling and reverse property, with $\eta$ the doubling order and $\nu$ the reverse doubling order, we obtain, writing $d = d(x,y)$:
	\begin{equation}
		\begin{cases}
			\frac{1}{\Vol{x}{\sqrt{t}}} \leq c\frac{1}{\Vol{x}{d}} \parenfrac{d}{\sqrt{t}}^\eta & 0 < t \leq d^2 \\
			\frac{1}{\Vol{x}{\sqrt{t}}} \leq c\frac{1}{\Vol{x}{d}} \parenfrac{d}{\sqrt{t}}^\nu & t > d^2
		\end{cases},
	\end{equation}
	
	\noindent and so:
	
	\begin{equation}
		i_s(x,y) \leq C \frac{1}{\Vol{x}{d}} 
		\left( 
			d^\eta \int_0^{d^2} t^{\frac{s - \eta}{2} - 1}e^{-\frac{d^2}{5t}}dt + 
			d^\nu \int_{d^2}^\infty	 t^{\frac{s - \nu}{2} - 1}e^{-\frac{d^2}{5t}} dt
		\right).
	\end{equation}
	
	Then provided $\nu > s$, we have:
	
	\begin{equation}
		\int_{d^2}^\infty	 t^{\frac{s - \nu}{2} - 1}e^{-\frac{d^2}{ct}} dt \leq \int_{d^2}^\infty t^{\frac{s - \nu}{2} - 1} dt = \frac{2}{s - \nu}d^{s - \nu}.
	\end{equation}
	
	For the other integral, we make the change of variable $t = d^2/u$, obtaining:
	
	\begin{equation}
		\int_0^{d^2} t^{\frac{s - \eta}{2} - 1}e^{-\frac{d^2}{ct}}dt = d^{s - \eta} \int_1^\infty u^{\frac{\eta - s}{2} - 1} e^{-\frac{u}{c}} du.
	\end{equation}
	
	This integral is convergent and equal to a constant that depends only on $s$ and $c$. Then, we have for every $x,y \in M$:
	
	\begin{equation}
		i_s(x,y) \leq C \frac{d(x,y)^s}{\Vol{x}{d(x,y)}}
	\end{equation}
\end{proof}

\subsection{Estimation of the Bessel potential}
	Define the Bessel potential for $\lambda > 0$, $s > 0$ to be the operator $G_{s,\lambda} = \left( \Delta + \lambda^2 \right)^{-s/2}$ on $L^2(M, \mu)$. It is, by the spectral theorem, a bounded operator, and, similarly to the case of the Riesz potential, admits for kernel:
	\begin{equation}
		g_{s}^\lambda (x,y) = \frac{1}{\Gamma\left(\frac{s}{2}\right)} \int_0^\infty t^{s/2-1} e^{-\lambda^2 t} p_t(x,y) \diff t,
	\end{equation}
	
	\noindent provided that $g_{s}^\lambda$ is finite for all $x,y \in M$.

\begin{proposition}\label{prop:bessel estimate}
	Let $(M,g,\mu)$ be a complete weighted manifold satisfying $\FK{R}$ and $\rd{R}$. If $\lambda > 0$ is such that $\lambda R \geq 1$, then for any $s < \nu$, there is a constant $C > 0$, depending only on s and on the Faber-Krahn constants, such that for all $x,y \in X$ with $d(x,y) \leq R$, we have:
	\begin{equation}
		g_{s}^{\lambda}(x,y) \leq C \frac{d(x,y)^s}{\Vol{x}{d(x,y)}}.
	\end{equation}		
	
\end{proposition}

\begin{proof}

We have:

\begin{equation}
	g_s^\lambda(x,y) = \frac{1}{\Gamma\parenfrac{s}{2}} \int_0^\infty t^{\frac{s}{2}-1} e^{-\lambda^2 t} p_t(x,y) dt,
\end{equation}

\noindent and we split this integral into three, integrating on $(0, d^2)$, $(d^2, R^2)$ and $(R^2, +\infty)$. 

We use proposition \ref{pr:estimate on the heat kernel}. The same calculations as in the proof for the Riesz potential yields the estimate:

\begin{equation}
	\int_0^{d^2} t^{\frac{s}{2} -1} e^{-\lambda^2 t} p_t(x,y) dt \leq C \frac{d(x,y)^s}{\Vol{x}{d(x,y)}}
\end{equation}

When $d \leq \sqrt{t} \leq R$, we have by the $R$-reverse doubling that $\Vol{x}{d} \leq a \parenfrac{d}{\sqrt{t}}^\nu \Vol{x}{\sqrt{t}}$, thus: 

\begin{align*}
	\int_{d^2}^{R^2} t^{\frac{s}{2} - 1} e^{-\lambda^2 t} p_t(x,y) dt 
	&\leq 
	C \int_{d^2}^{R^2} t^{\frac{s}{2} - 1} e^{-\lambda^2 t} \frac{1}{\Vol{x}{\sqrt{t}}} e^{-\frac{d^2}{5t}} dt \\
	&\leq
	C \frac{d^{\nu}}{\Vol{x}{d}} \int_{d^2}^{R^2} t^{\frac{s - \nu}{2} - 1} e^{-\lambda^2 t} e^{-\frac{d^2}{5t}} dt \\
	&\leq
	C \frac{d^{\nu}}{\Vol{x}{d}} \int_{d^2}^{R^2} t^{\frac{s - \nu}{2} - 1} e^{-\lambda^2 t} dt \\
	&\leq 
	C \frac{d^{\nu}}{\Vol{x}{d}} \frac{2}{s - \nu}\left( R^{s - \nu} - d^{s - \nu}\right),
\end{align*}

\noindent and since $s - \nu < 0$ we have:

\begin{equation}
	\int_{d^2}^{R^2} t^{\frac{s}{2} - 1} e^{-\lambda^2 t} p_t(x,y) dt \leq C \frac{d(x,y)^s}{\Vol{x}{d(x,y)}}.
\end{equation}

Now for $t \geq R^2$, we simply have $\Vol{x}{R} \leq \Vol{x}{\sqrt{t}}$. Thus:
\begin{equation*}
	\int_{R^2}^\infty t^{\frac{s}{2} -1} e^{-\lambda^2 t} p_t(x,y) dt 
	\leq 
	C \frac{1}{\Vol{x}{R}} \int_{R^2}^\infty t^{\frac{s}{2}-1} e^{-\lambda^2 t}  e^{-\frac{d^2}{5t}} dt,
\end{equation*}

\noindent then, since $d \leq R$, by using the reverse doubling we obtain $\Vol{x}{d} \leq \parenfrac{d}{R}^\nu \Vol{x}{R}$. Moreover, we have $t^{\frac{s}{2} - 1} e^{-\lambda^2 t} \leq c_s \lambda^{2 - s} e^{-\frac{\lambda^2}{2}t}$, thus we have:

\begin{align*}
	\int_{R^2}^\infty t^{\frac{s}{2} -1} e^{-\lambda^2 t} p_t(x,y) dt 
	&\leq C
	\parenfrac{d}{R}^\nu \frac{\lambda^{2 - s}}{\Vol{x}{d}}\int_{R^2}^\infty e^{-\frac{\lambda^2}{2}t} dt \\ 
	&\leq C \parenfrac{d}{R}^\nu \frac{\lambda^{-s}}{\Vol{x}{d}} e^{-\frac{(\lambda R)^2}{2}}
\end{align*}

Then, since $\lambda R \geq 1$, and $f(t) = t^{-s} e^{-\frac{t^2}{2}}$ is decreasing, we have $\lambda^{-s} e^{-\frac{(\lambda R)^2}{2}} \leq R^s e^{-\frac{1}{2}}$, which leads to:

\begin{equation}
	\int_{R^2}^\infty t^{\frac{s}{2} -1} e^{-\lambda^2 t} p_t(x,y) dt 
	\leq C \parenfrac{R}{d}^{s - \nu} \frac{d^s}{\Vol{x}{d}},
\end{equation}

\noindent then $s - \nu < 0$ and $d < R$, thus we have:

\begin{equation}
	\int_{R^2}^\infty t^{\frac{s}{2} -1} e^{-\lambda^2 t} p_t(x,y) dt 
	\leq C\frac{d(x,y)^s}{\Vol{x}{d(x,y)}} 
\end{equation}
\end{proof}	

\section{Proof of the main results}\label{sec:Proof}

Let $(M,g,\mu)$ be a weighted Riemannian manifold. Let $V\in L^1_{loc}(M, \diff\mu)$, $V \geq 0$, for any $R > 0$ and $p \geq 1$, we define $N_p(V)$ and $N_{p,R}(V)$ as in \eqref{eq:Morrey} and \eqref{eq:def morrey norm}. Notice that $N_p(V) = M_{2p}(V^p)^{1/p}$.

Though we can deduce theorem \ref{th:Fefferman-Phong generalized} as a special case of \ref{th:Weak_positivity}, we start by giving a separate, simpler proof of it. The general idea behind the proof of both theorems remains the same, but in the case of theorem \ref{th:Weak_positivity}, much more care will be required in establishing the bounds on the norm of certain operators.

\subsection{Proof the global inequality (Theorem \ref{th:Fefferman-Phong generalized})}

	We assume here that $\mu$ is reverse doubling of order $\nu$, with $\nu > 1$, and we will show later on that this implies the general result.
	
	Given $\varphi \in L^2(M)$, we first estimate $\left\|\Delta^{-1/2}\left(V^{1/2}\varphi\right)\right\|_2$. By proposition \ref{pr:riesz_potential},	for any non-negative, measurable function $f$, we have:
	\begin{equation*}
		\Delta^{-\frac{1}{2}} f(x) \leq C \int_M \frac{d(x,y)}{\Vol{x}{d(x,y)}} f(y) d\mu(y).
	\end{equation*}
	
	Let $T$ be the operator defined by the kernel $K(x,y) = \frac{d(x,y)}{\Vol{x}{d(x,y)}}$. Since $M$ is a doubling space, applying corollary \ref{cor:riesz_bounded}, we have that:
	\begin{equation*}
		\|Tf\|_2 \leq C \|M_1 f\|_2,
	\end{equation*}
	
	\noindent and so:
	\begin{equation*}
		\|\Delta^{-\frac{1}{2}} f\| \leq \|M_1 f\|_2.
	\end{equation*}
	
	It follows that:
	\begin{equation*}
		\left\|\Delta^{-1/2}\left(V^{1/2}\varphi\right)\right\|_2 \leq C\left\|M_1\left(V^{1/2}\varphi\right)\right\|_2.
	\end{equation*}
	
	Then, using the Hölder inequality, we have, with $q = 2p$, $\frac{1}{q} + \frac{1}{q'} = 1$:
	\begin{equation*}
		M_1 \left(V^{1/2}\varphi\right) \leq M_q \left(V^{q/2}\right)^{\frac{1}{q}} M_0 \left(|\varphi|^{q'}\right)^\frac{1}{q'} \leq N_p(V)^\frac{1}{2} M_0 \left( |\varphi|^{q'}\right)^\frac{1}{q'},
	\end{equation*}
	
	\noindent since $N_p(V) = M_{2p}(V^p)^\frac{1}{p}$. And by the $L^{2/q'}$ boundedness of the Hardy-Littlewood maximal function, we obtain that:
	\begin{align*}
		\left\|M_1 \left(V^{1/2}\varphi\right)\right\|_2 
		&\leq N_p(V)^\frac{1}{2} \left\|M_0 \left( |\varphi|^{q'}\right)^\frac{1}{q'}\right\|_2 \\
		&\leq N_p(V)^\frac{1}{2} \left\|M_0 \left( |\varphi|^{q'}\right)\right\|_{\frac{2}{q'}}^\frac{1}{q'} \\
		&\leq C_{p} N_p(V)^\frac{1}{2} \|\varphi\|_2.
	\end{align*}
	
	And so:
	\begin{equation}\label{eq:global adjoint L2 norm}
		\left\|\Delta^{-1/2}\left(V^{1/2}\varphi\right)\right\|_2 \leq C_p N_p(V)^\frac{1}{2} \|\varphi\|_2,
	\end{equation}
	
	\noindent and $\Delta^{-1/2}\left(V^{1/2} \cdot\right)$ is a bounded linear operator on $L^2$. Its adjoint is $V^{1/2}\Delta^{-1/2}$, and for any $\psi \in \SmoothComp{M}$, if we let $\varphi = \Delta^{1/2}\psi$, $\varphi \in L^2$ and $\int_M V\psi^2 d\mu = \int_M \left|V^{1/2}\Delta^{-1/2}\varphi\right|^2 d\mu$. By \eqref{eq:global adjoint L2 norm} we get:
	
	\begin{equation}
		\int_M V \psi^2 d\mu \leq C_p N_p(V)^{1/2} \|\varphi\|_2^2 = C_p N_p(V)\|\nabla\psi\|^2_2.
	\end{equation}

\subsection{Proof of the local inequality (Theorem \ref{th:Weak_positivity})}

%
%

We again make a technical hypothesis on the reverse doubling order, proving the following result:

\begin{theorem}\label{th:c_lambda_estimates}
	Let $(M,g,\mu)$ be a complete weighted Riemannian manifold satisfying $\FK{R}$ for some $R > 0$, and $\rd{R}$ for some $\nu > 1$. Then, for any $p > 1$, there is some constant $C_p$ depending only on the Faber Krahn constants and $p$, such that for any non-negative, locally integrable $V$, and any $\psi \in \SmoothComp{M}$, the following inequality holds:	
	\begin{equation}
		\int_M V\psi^2 d\mu \leq C_p N_{p,R}(V) \left( \int_M |\nabla \psi|^2 d\mu + \frac{1}{R^2} \int_M \psi^2 d\mu \right)
	\end{equation}
\end{theorem}
%
%
%

We will show afterwards how to remove this hypothesis to obtain theorem \ref{th:Weak_positivity}

\subsubsection{Proof of Theorem \ref{th:c_lambda_estimates}}

	Given $\lambda > 0$ such that $\lambda R \geq 1$, we let $g^\lambda = g_1^\lambda$ be the kernel of the Bessel potential $G^\lambda = \left(\Delta + \lambda^2 \right)^{-\frac{1}{2}}$. By proposition \ref{prop:bessel estimate}, we have $g^\lambda (x,y) \leq \frac{d(x,y)}{\Vol{x}{d(x,y)}}$, for $\lambda d(x,y) < 1$. We let:
	\begin{equation}\label{eq:def T_1 T_2}
		T_1 \psi(x) = \int_{d(x,y) \leq R} g^{\lambda}(x,y) V^{\frac{1}{2}}(y) \psi(y) d\mu(y),\quad T_2 \psi(x) = \int_{d(x,y) > R} g^{\lambda}(x,y) V^{\frac{1}{2}}\psi(y) d\mu(y)
	\end{equation}
	
	By corollary \ref{cor:riesz_bounded}, we have $\|T_1\psi\|_p \leq C_p \left\|M_{1,R} V^{\frac{1}{2}}\psi\right\|_p$, and the rest follows as in the global case. To estimate $\|T_2\|$, we can study the operator $T_2 T_2^*$, with kernel $a(x,z)$ defined as:
	\begin{equation*}
		a(x,z) = \int_M g^{\lambda}(x,y) \charaset_{\set{d(x,y) > R}} |V(y)| \charaset_{\set{d(y,z) > R}}g^{\lambda}(y,z) d\mu(y),
	\end{equation*}
	
	\noindent where we recall $\charaset_E$ to be the characteristic function of the set $E$. We then apply the Schur test to $T_2T_2^*$: being a symetric operator, it will be bounded on $L^2$ if the integral:	
	\begin{equation*}
		\int_M |a(x,z)|d\mu(z)
	\end{equation*}
	
	\noindent is uniformely bounded with respect to $x$. Given that we have $g^\lambda(y,z) = \int_0^\infty \frac{e^{-\lambda^2 t}}{\sqrt{\pi t}} p_t(y,z) dt$, as well as $\int_M p_t(y,z) d\mu(z) \leq 1$, we calculate:
	\begin{equation*}
		\int_M g^\lambda(y,z) d\mu(z) \leq \int_0^\infty \frac{e^{-\lambda^2 t}}{\sqrt{\pi t}} dt = \frac{1}{\lambda},
	\end{equation*}
	
	\noindent but then:
	\begin{align*}
		\int_M |a(x,z)|d\mu(z) &\leq \int_M \int_M g^{\lambda}(x,y) \charaset_{\set{d(x,y) > R}} |V(y)| \charaset_{\set{d(y,z) > R}}g^{\lambda}(y,z) d\mu(y)d\mu(z) \\
		&\leq \int_{M\setminus B(x,R)} g^{\lambda}(x,y) \int_{M\setminus B(z,R)} g^{\lambda}(y,z) d\mu(z) |V(y)| d\mu(y)d\mu(z),
	\end{align*}
	
	\noindent and so we get:
	\begin{equation*}
		\int_M |a(x,z)| d\mu(z) \leq \frac{1}{\lambda}\int_{M\setminus B(x,R)} g^\lambda(x,y) |V(y)| d\mu(y),
	\end{equation*}
	
	\noindent or:
	\begin{equation}
		\int_M |a(x,z)| d\mu(z) \leq \frac{1}{\lambda}\int_0^{+\infty} \frac{e^{-\lambda^2 t}}{\sqrt{\pi t}} \int_{M\setminus B(x,R)}  p_t(x,y) |V(y)|d\mu(y) dt
	\end{equation}
	
	To estimate this integral, we estimate $\int_{M\setminus B(x,R)} p_t(x,y) |V(y)| d\mu(y)$ by distinguishing the cases $t \geq R^2$ and $t < R^2$. 
	
	For any $x\in M$, $r \geq R$, $p \geq 1$, we have:
	\begin{equation}\label{eq:estimate on int V}
		\int_{B(x,r)} |V(y)|d\mu(y) \leq \frac{C}{R^2} \Vol{x}{r} N_{p,R}(V)
	\end{equation}
	
	Indeed, we cover $B(x,r)$ by a family $B_i$ of balls of radius $R$ with center in $B(x,r)$ such that the balls with half the radius are pairwise disjoints. Then we have:
	\begin{align*}
		\int_{B(x,r)} |V(y)| d\mu(y) &\leq 
		\sum_i \int_{B_i} |V(y)| d\mu(y) \\
		&\leq
		C \sum_i \mu\left(\frac{1}{2}B_i \right) \fint_{B_i} |V(y)| d\mu(y) \\
		&\leq
		\frac{C}{R^2} \sum_i \mu\left(\frac{1}{2}B_i \right) \left(R^2\fint_{B_i} |V(y)|^p d\mu(y)\right)^\frac{1}{p} \\
		&\leq \frac{C}{R^2} \Vol{x}{r + \frac{R}{2}} N_{p,R}(V) \\
		&\leq \frac{C}{R^2} \Vol{x}{r} N_{p,R}(V).
	\end{align*}
	
	For all $t \geq R^2$, we use the corresponding estimate of proposition \ref{pr:estimate on the heat kernel}, and get a constant $\alpha > 0$ such that:
	
	\begin{equation*}
		\int_{M\setminus B(x,R)} p_t(x,y)|V(y)| d\mu(y) \leq
		\frac{Ce^{\alpha\frac{t}{R^2}}}{\Vol{x}{R}} \int_{M\setminus B(x,R)} e^{-\frac{d(x,y)^2}{5t}}|V(y)| d\mu(y).
	\end{equation*}
	
	By writhing $e^{-\frac{d(x,y)^2}{5t}} = \int_{d(x,y)}^{+\infty} \frac{2r}{5t} e^{-\frac{r^2}{5t}} dr$, we have:		
	\begin{equation*}
		\int_{M\setminus B(x,R)} e^{-\frac{d(x,y)^2}{5t}} |V(y)|d\mu(y) = \int_R^\infty e^{-\frac{r^2}{5t}} \frac{2r}{5t} \left(\int_{B(x,r)\setminus B(x,R)} |V| d\mu \right) dr,
	\end{equation*}
	
	\noindent then using \eqref{eq:estimate on int V} we obtain
	\begin{equation*}
		\int_{M\setminus B(x,R)} e^{-\frac{d(x,y)^2}{5t}} |V(y)| d\mu(y) 
		\leq 
		\frac{C}{R^2}N_{p,R}(V) \int_R^\infty e^{-\frac{r^2}{5t}} \frac{2r}{5t}\Vol{x}{r}dr,
	\end{equation*}
	
	\noindent then by the $R$-doubling, using \eqref{eq:quasi_doubling}, there is a constant $\beta > 0$ that depends on the doubling constant such that $\Vol{x}{r} \leq \Vol{x}{R} e^{\beta \frac{r}{R}}$, thus:
	\begin{equation*}
		\int_{M\setminus B(x,R)} e^{-\frac{d(x,y)^2}{5t}} |V(y)| d\mu(y)
		\leq
		\frac{C}{R^2}N_{p,R}(V) \Vol{x}{R} \int_R^\infty \frac{2r}{5t} e^{-\frac{r^2}{5t} + \beta\frac{r}{R}} dr,
	\end{equation*}
	
	\noindent and we then can find a constant $\gamma > 0$ such that $e^{-\frac{r^2}{5t} + \beta \frac{r}{R}} \leq e^{-\frac{r^2}{10t} + \gamma\frac{t}{R^2}}$. As a result, we get that:	
	\begin{equation*}
		\int_R^\infty \frac{2r}{5t}e^{-\frac{r^2}{10t}} dr = 2 e^{-\frac{R^2}{10t}}.
	\end{equation*}

		To conclude we obtain that for all $t \geq R^2$
	\begin{equation}
		\int_{M\setminus B(x,r)} p_t(x,y) |V(y)| d\mu(y) \leq C \frac{N_{p,R}(V)}{R^2} e^{-\frac{R^2}{10t} + 2\gamma\frac{t}{R^2}}.
	\end{equation}
	
	For $t \leq R^2$, we obtain in the same way:
	
	\begin{align*}
		\int_{M\setminus B(x,R)} p_t(x,y) |V(y)| d\mu(y) 
		&\leq 
		\frac{c}{\Vol{x}{\sqrt{t}}} \frac{N_{p,R}(V)}{R^2} \int_{R}^\infty e^{-\frac{r^2}{5t}} \frac{2r}{5t} \Vol{x}{r} dr \\
		&\leq
		\frac{c}{\Vol{x}{\sqrt{t}}} \Vol{x}{R} \frac{N_{p,R}(V)}{R^2} \int_R^\infty e^{-\frac{r^2}{5t}} \frac{2r}{5t} e^{\beta \frac{r}{R}} dr\\
		&\leq
		c \parenfrac{R}{\sqrt{t}}^\nu \frac{N_{p,R}(V)}{R^2} e^{-\frac{R^2}{10t} + \gamma \frac{t}{R^2}},
	\end{align*}
	
	\noindent and finally we obtain:	
	\begin{equation}
		\int_{M\setminus B(x,R)} p_t(x,y)|V(y)| d\mu(y) \leq C \left(\max\left(\frac{R}{\sqrt{t}}, 1 \right) \right)^\nu \frac{N_{p,R}(V)}{R^2} e^{-\frac{R^2}{10t} + 2\gamma \frac{t}{R^2}}.
	\end{equation}
	
	Thus we get the majoration:	
	\begin{equation}
		\int_M |a(x,z)| d\mu(z) \leq \frac{C}{\lambda}\frac{N_{p,R}(V)}{R^2} \int_0^\infty \left(\max\left(\frac{R}{\sqrt{t}}, 1 \right) \right)^\nu  e^{-\frac{R^2}{10t} + 2\gamma \frac{t}{R^2}} e^{-\lambda^2 t} \frac{dt}{\sqrt{\pi t}},
	\end{equation}
	
	\noindent which by a change of variable $t = R^2 u$, transform into:
	
	\begin{equation}
		\int_M |a(x,z)| d\mu(z) \leq \frac{C}{\lambda}\frac{N_{p,R}(V)}{R} \int_0^\infty \left(\max\left(\frac{1}{\sqrt{u}}, 1 \right) \right)^\nu  e^{-\frac{1}{10u} + (2\gamma - \lambda^2 R^2)u} \frac{du}{\sqrt{\pi u}},
	\end{equation}
	
	\noindent and if $\lambda R \geq \sqrt{3\gamma} = \kappa > 1$ we obtain:
	\begin{equation}
		\int_M a(x,z) dz \leq C N_{p,R}(V).
	\end{equation}
	
	Thus, by the Schur test, $\|T_2 T_2^*\|_{L^2 \rightarrow L^2} \leq C N_{p,R}(V)$, and $\|T_2\|_{L^2 \rightarrow L^2} \leq C N_{p,R}(V)^\frac{1}{2}$. Then, we have, for all $\lambda \geq \frac{\kappa}{R}$:
	\begin{equation}
		\int_M V \psi^2 d\mu \leq C\left( \int_M |\nabla \psi|^2 d\mu + \lambda^2 \int_M \psi^2 d\mu \right),
	\end{equation}
	
	\noindent and in particular:
	\begin{equation}
		\int_M V \psi^2 d\mu \leq C\kappa^2\left( \int_M |\nabla \psi|^2 d\mu + \frac{1}{R^2} \int_M \psi^2 d\mu \right).
	\end{equation}

\subsubsection{Proof of theorem \ref{th:Positive lambda_1}}

We now suppose that $\lambda_1(M) > 0$. Then the previous results can be strenghtened to prove theorem \ref{th:Positive lambda_1}.

\begin{proof}
	We apply theorem \ref{th:Weak_positivity}, and use that $\lambda_1(M) \int_M \psi^2 \diff\mu \leq \int_M |\nabla\psi|^2\diff\mu$. Then we obtain:
	
	\begin{equation*}
		\llangle V\psi,\psi\rrangle \leq C_p N_{p,R}(V) \left(1 + \frac{1}{\lambda_{1}(M)R^2}\right) \int_M |\nabla\psi|^2\diff\mu,
	\end{equation*}
	
	\noindent which gives:
	
	\begin{equation*}
		\frac{\lambda_1(M) R^2}{C_p N_{p,R}(V)(1 + \lambda_1(M) R^2)} \int_M V\psi^2 \diff\mu \leq \int_M |\nabla \psi|^2 \diff\mu,
	\end{equation*}
	
	\noindent and:
	
	\begin{equation*}
		\frac{\lambda_1(M) R^2}{2C_p N_{p,R}(V)(1 + \lambda_1(M) R^2)}  \int_M V\psi^2 \diff\mu + \frac{\lambda_1(M)}{2} \int_M \psi^2 \diff\mu \leq \int_M |\nabla \psi|^2 \diff\mu.
	\end{equation*}
	
	Then, for any $V$, we have:
	
	\begin{equation}
		\llangle V\psi, \psi\rrangle \leq \frac{C_p N_{p,R}(V)(1 + \lambda_1(M)R^2)}{\lambda_1(M)R^2}  \left( \|\nabla\psi\|^2 - \frac{\lambda_1(M)}{2} \|\psi\|^2 \right),
	\end{equation}
	
	\noindent which is \eqref{eq:Positive lambda_1}.
\end{proof}

\subsection{Proof of theorem \ref{th:lower_bound estimates}}

Let $C_p$ be the constant of theorem \ref{th:Weak_positivity}. We let 

\begin{equation}
	L = \sup_{x, \delta} \left(2C_p \left(\fint_{B(x,\delta)} V^p \diff\mu \right)^{1/p} - \delta^{-2} \right).
\end{equation}

Then we have:

\begin{align*}
	\left(\fint_{B(x,\delta)} V^p \diff\mu \right)^{1/p} 
	&\leq \frac{L + \delta^{-2}}{2C_p}, \\
	\left(M_{2p,\delta} (V^p)(x) \right)^{1/p} 
	&\leq \frac{\delta^2 L + 1}{2C_p}.
\end{align*}

Take $\delta = L^{-1/2}$, then $N_{p,\delta}(V) \leq \frac{1}{C_p}$. Then by theorem \ref{th:Weak_positivity} we have:
\begin{equation}
	\llangle V\psi, \psi\rrangle - \|\nabla\psi\|_2^2 \leq L\|\psi\|^2,
\end{equation}

\noindent thus:
\begin{equation}
	- \lambda_1(\Delta - V) \leq \sup_{x, \delta} \left(2C_p \left(\fint_{B(x,\delta)} V^p \diff\mu \right)^{1/p} - \delta^{-2} \right).
\end{equation}

Meanwhile, let $r < \lambda^{-1} \leq R$, and define $f_r: [0, \infty) \goesto [0, +\infty)$ by $f(t) = r$ if $t \leq r$, $f(t) = 2r - t$ if $t \in (r, 2r]$ and $f_r(t) = 0$ if $t > 2r$. Then for $o \in M$, $\psi = f_r(d(o,x))$. $\psi$ is a Lipschitz function with compact support, and we have, by $\doubling{R}$:

\begin{align*}
	\lambda_1(\Delta - V) &\leq \frac{\|\nabla \psi\|^2 - \int_M V\psi^2 \diff\mu}{\|\psi\|^2} \\
	&\leq \frac{\Vol{x}{2r}}{r^2\Vol{x}{r}} - \fint_{B(x,r)} V\diff\mu \\
	&\leq A r^{-2} -  \fint_{B(x,r)} V\diff\mu \\
	&\leq (r/\sqrt{A})^{-2} -  A^{-1 - \eta/2} \fint_{B(x,r/\sqrt{A})} V \diff\mu,
\end{align*}

this for all $r > 0$. Thus:

\begin{equation}
	-\lambda_1(\Delta - V) \geq \sup_{x,\delta} \left( A^{-1 - \eta/2} \fint_{B(x,\delta)} V \diff\mu - \delta^{-2}\right).
\end{equation}

\subsection{Removing the dependancy on reverse doubling}

Let $M$ be a manifold satisfying $\FK{}$. We consider $\tilde{M} = \R\times M$, $(\tilde{M},\tilde{g}, \tilde{\mu})$ the product Riemannian manifold: $\tilde{g} = \diff x^2 + g$, $\diff\tilde{\mu} = \diff x \diff\mu$. For $V \in L_{loc}^1(M)$ we define $\tilde{V}(x,m) = V(m)$. We write $\tilde{\Delta}$ for the laplacian on $(\tilde{M},\tilde{g},\tilde{\mu})$, and $\Delta$ for the laplacian on $(M,g,\mu)$. The Morrey norm in $\tilde{M}$ is written $\tilde{N}_{p,R}$.

We have:

\begin{proposition}
	$(\tilde{M},\tilde{g},\tilde{\mu})$ satisfies the following properties:
	
	\begin{enumerate}
		\item If $\mu$ is $R$-doubling, then $\tilde{\mu}$ is $R$-doubling, and $R$-reverse doubling with order $\nu > 1$.
		\item The heat kernel of $\tilde{M}$ is $\tilde{p}_t((x,m),(y,n)) = \frac{1}{\sqrt{4\pi t}} e^{-\frac{|x - y|^2}{4t}} p_t(m,n)$.
		\item If $M$ satisfies $\FK{R}$, then there is some $\theta \in (0,1)$ such that $\tilde{M}$ satisfies $\FK{\theta R}$. $\theta$ depends only on the Faber Krahn constants.
		\item $\lambda_1(\tilde{\Delta} - \tilde{V}) = \lambda_1(\Delta - V)$
		\item If $\mu$ is $R$-doubling, then there are two constants $c, C$ which depends only on the doubling constant, such that $cN_{p,R}(V) \leq \tilde{N}_{p,R}(\tilde{V}) \leq C N_{p,R}(V)$
	\end{enumerate}
\end{proposition}

\begin{proof}\

$1.$ For $E\subset \R$ measurable, we denote $|E|$ the usual lebesgue measure of $E$. We have:

\begin{equation}\label{eq:volume comparison}
	|(-r/2,r/2)|\mu(B(m, r/2)) \leq \tilde{\mu}(\tilde{B}((x,m), r)) \leq |(-r, r)|\mu(B(m,r)).
\end{equation}

From this, with $r \leq R$ we immediately get $\tilde{\mu}(\tilde{B}((x,m), 2r)) \leq 4 A^2 \tilde{\mu}(\tilde{B}((x,m), r))$, with $A$ the $R$-doubling constant of $\mu$. Moreover, since $\mu$ is $R$-doubling, it is $R$-reverse doubling, with reverse doubling order $\nu > 0$. Then, we have, for $r < r' < \theta R$:

\begin{align*}
	\frac{\tilde{\mu}(\tilde{B}((x,m),r'))}{\tilde{\mu}(\tilde{B}((x,m),r))} &\geq \frac{r'}{2r}\frac{ \mu(B(m,r'/2))}{\mu(B(m,r))} \\
	&\geq \frac{1}{2A} \frac{r'}{r} \frac{\Vol{m}{r'}}{\Vol{m}{r}} \\
	&\geq \frac{a}{2A} \parenfrac{r'}{r}^{1+\nu}
\end{align*}

Thus $\tilde{\mu}$ is reverse doubling of order $\tilde{\nu} = 1 + \nu > 1$.

$2., 4.$ We have $\tilde{\Delta} = -\frac{\diff^2}{\diff x^2} + \Delta$. Thus $\tilde{p}_t((x,m),(y,n)) = \frac{1}{\sqrt{4\pi t}} e^{-\frac{|x - y|^2}{4t}} p_t(m,n)$, and the spectrum of $\tilde{\Delta} - \tilde{V}$ is:

\begin{equation*}
	Sp(\tilde{\Delta} - \tilde{V}) = \set*{\lambda + \lambda';\quad \lambda \in Sp(\Delta - V), \lambda' \geq 0}.
\end{equation*}

Thus the infimum of the spectrum of $\tilde{\Delta} - \tilde{V}$ is the infimum of the spectrum of $\Delta - V$.

$3.$ We use proposition \ref{pr:heat + doubling implies FKR}.

$5.$ We use \eqref{eq:volume comparison}. Using that $\int_{\tilde{B}} \tilde{V} \diff\tilde{\mu} \leq 2r \int_B V\diff\mu$, we have:

\begin{equation*}
	\frac{r^{2p}}{\tilde{\mu}(\tilde{B}((x,m),r)} \int_{\tilde{B}} \tilde{V}^p \diff\tilde\mu \leq \frac{r^{2p}}{(r/2) \Vol{m}{r/2}} 2r \int_B V^p \diff\mu.
\end{equation*}

Then by $R$ doubling $\tilde{N}_{p,R}(\tilde{V}) \leq 4A N_{p,R}(V)$. The other inequality is obtained in a similar same way.
\end{proof}

\begin{proof}[Proof of theorem \ref{th:Weak_positivity}]

From the points $1., 3.$ of the above proposition, if $(M,g,\mu)$ is a manifold satisfying $\FK{R}$, then there is some $\theta \in (0, 1)$, depending only on the Faber Krahn constants, such that $(\tilde{M},\tilde{g},\tilde{\mu})$ satisfies $\FK{\theta R}$ and $\rd{R}$, with $\nu > 1$. Then we can apply \ref{th:Weak_positivity} to $\tilde{M}$: there is a constant $\tilde{C}_p$ such that if $\tilde{V}$ is such that $\tilde{C}_p \tilde{N}_{p,R}(\tilde{V}) \leq 1$, then $\lambda_1(\tilde{\Delta} - \tilde{V}) \geq -\frac{1}{\theta^2 R^2}$.

Using $5.$, then there is a constant $C_p > 0$ such that $C_p N_{p,R}(V) \geq \tilde{C_p} \tilde{N}_{p,R}(\tilde{V})$. Then since $\lambda_1(\Delta - V) = \lambda_1(\tilde{\Delta} - \tilde{V})$, if $\C_p N_{p,R}(V) \leq 1$, then $\lambda_1(\Delta - V) \geq -\frac{1}{\theta^2 R^2}$. For an arbitrary $V \geq 0$, locally integrable, with $N_{p,R}(V) < +\infty$, we can apply the above to $V/C_p N_{p,R}(V)$, then for any $\psi \in \SmoothComp{M}$:

\begin{equation}
	\frac{1}{C_p N_{p,R}(V)}\int_M V \psi^2 \diff \mu \leq \frac{1}{\theta^2} \int_M \left( |\nabla \psi|^2 + \frac{1}{ R^2} \psi^2 \right) \diff\mu,
\end{equation}

\noindent which is \eqref{eq:Weak_majoration}.
\end{proof}

\section{Hardy inequality}\label{sec:Hardy}
For some point $o \in M$, the $L^2$ Hardy inequality:

\begin{equation}
	\forall \psi \in \SmoothComp{M},\; \int_M \frac{\psi(x)^2}{d(o,x)^2} \diff\mu(x) \leq C \int_M |\nabla \psi(x)|^2 \diff\mu(x)
\end{equation}

\noindent is equivalent to the positivity of the operator $\Delta - V$, with $V(x) = \frac{1}{C}d(o,x)^{-2}$. Moreover, we have:

\begin{proposition}
	Let $(M, g, \mu)$ be a weighted Riemannian manifold, $R \in (0, \infty]$. If $\mu$ satisfies $\doubling{R}$ and $\rd{R}$, with $\nu > 1$, then for any $p \in (1, \nu/2)$, there is a constant $K_p < \infty$ such that for all $r < R$ we have:
	
	\begin{equation}\label{eq:hardy_good_potential}
		r^2 \left( \fint_{B(x,r)} d(o,y)^{-2p} \diff \mu \right)^{1/p} \leq K_p.
	\end{equation}
\end{proposition}

	\begin{proof}
	We let $\rho(y) = d(o,y)$, $B= B(x,r)$, for $r < R$.
	
	If $r \leq \rho(x)/2$, then for $y \in B(x,r)$, $\rho(y) \geq \rho(x) - r \geq \rho(x)/2 \geq r$. Then:
	
	\begin{equation*}
		\int_B \rho(y)^{-2p} \diff\mu \leq r^{-2p} \mu(B).
	\end{equation*}
	
	If $r > \rho(x)/2$, then $B(x,r) \subset B(o,3r)$, and:
	
	\begin{align*}
		\int_B \rho^{-2p} \diff\mu &\leq \int_{B(o,3r)} \rho^{-2p} \diff\mu \\
		&\leq \int_0^\infty (2p - 1)t^{-2p - 1} \Vol{o}{\min(t, 3r)} \diff t \\
		&\leq \int_0^{3r} a^{-1} (2p - 1)t^{\nu - 2p - 1} (3r)^{-\nu} \Vol{o}{3r} \diff t + r^{-2p} \Vol{o}{3r} \\
		&\leq \left( \frac{1}{3^{3p}a} \frac{2p-1}{\nu - 2p} + 1\right) r^{-2p}\Vol{o}{3r} \\
		&\leq C_p r^{-2p}\Vol{x}{r},
	\end{align*}
	
	since $\nu > 2p$, with the constant $C_p$ depending uniquely on $p$ and the doubling and reverse doubling constants.
\end{proof}

Then applying theorems \ref{th:Weak_positivity} and \ref{th:Fefferman-Phong generalized}, we immediately obtain:	

\begin{corollary}
	If $(M,g,\mu)$ satisfies $\FK{R}$ and $\rd{R}$ with $\nu > 2$, then there is a constant $C$ such that for any $\psi \in \Continuous_0^\infty(M)$, $o\in M$,

	\begin{equation}
		\int_M \frac{\psi(x)^2}{d(o,x)^2} \diff\mu(x) \leq C \left(\|\nabla \psi\|_2^2 +  \frac{1}{R^2}\|\psi\|_2^2 \right).
	\end{equation}
\end{corollary}

\begin{corollary}
	If $(M,g,\mu)$ satisfies $\FK{}$, $\rd{}$ with $\nu > 2$ then there is a constant $C$ such that:
	\begin{equation}\label{eq:Hardy_C}
		\forall \psi \in \SmoothComp{M},\; \int_M \frac{\psi(x)^2}{d(o,x)^2} \diff\mu(x) \leq C \int_M |\nabla\psi|^2 \diff\mu.
	\end{equation}
\end{corollary}

The second corollary being theorem \ref{th:Hardy}.

This time the condition on the reverse doubling order is not merely a technical hypothesis. It is, in fact, a necessary condition for the Hardy inequality to holds if we assume the measure $\mu$ to be doubling:

\begin{proposition}
	Let $(M,g,\mu)$ be a weighted Riemannian manifold, with $\mu$ a doubling measure, assume that there is a constant $\nu > 2$ such that for any $o \in M$, $\psi \in \SmoothComp{M}$, $M$ admits the Hardy inequality:
	
	\begin{equation}\label{eq:Hardy_nu}
		 \parenfrac{\nu-2}{2}^2 \int_M \frac{\psi(x)^2}{d(o,x)^2} \diff\mu(x) \leq \int_M |\nabla\psi|^2 \diff\mu,
	\end{equation}
	
	then $\mu$ satisfies $\rd{}$.
\end{proposition}

Note that that we can always write a Hardy inequality \eqref{eq:Hardy_C} in the form \eqref{eq:Hardy_nu} simply by chosing $\nu = 2 + 2\sqrt{1/C}$.

Using a method from \cite{Carron16, LiWang01}, we have:

\begin{proof}
	Take $0 < r < R$, define $f(t) = r^{-\frac{\nu- 2}{2}}$ for $0 \leq t \leq r$, $f(t) = t^{-\frac{\nu-2}{2}}$ for $r \leq t \leq R$, $f(t) = 2 R^{-\frac{\nu - 2}{2}} - R^{-\frac{\nu}{2}} t$ for $R \leq t \leq 2R$ and $f(t) = 0$ for $t \geq 2R$.
	
	When $r \leq t \leq R$, we have $f'(t)^2 = \left(\frac{\nu - 2}{2}\right)^2 \frac{f(t)^2}{t^2}$.	Then for some point $o \in M$ choose $\phi(x) = f(d(o,x))$, the Hardy inequality applied to $\varphi$ leads to:
	
	\begin{equation}
		\left(\frac{\nu-2}{2}\right)^2\int_{B(o,r)} \frac{\phi(x)^2}{d(o,x)^2} \diff\mu(x) \leq \int_{B(o,2R)\setminus B(o, R)} |\nabla \phi|^2 \diff\mu(x),
	\end{equation}
	
	then:
	
	\begin{equation}
		\left(\frac{\nu-2}{2}\right)^2 r^{-\nu} \Vol{o}{r} \leq R^{-\nu} \mu(B(o,2R)\setminus B(o,R)) \leq A R^{-\nu} \Vol{o}{R},
	\end{equation}
	
	using that $\mu$ is doubling. Thus there is some constant $a > 0$ such that:
	
	\begin{equation}
		a \left(\frac{R}{r}\right)^\nu \leq \frac{\Vol{o}{R}}{\Vol{o}{r}},
	\end{equation}
	
	and $\mu$ is reverse doubling of order $\nu > 2$.
\end{proof}

\bibliography{Bibliography.bib}
\bibliographystyle{plain}

\end{document}